\documentclass[11pt]{article}

\usepackage[cp1252]{inputenc}
\usepackage[english]{babel}
\usepackage[a4paper]{geometry}
\usepackage{amsmath,amsfonts,amssymb,amsthm,cases,upgreek}
\usepackage{empheq}
\usepackage{stmaryrd}

\SetSymbolFont{stmry}{bold}{U}{stmry}{m}{n}

\usepackage{dsfont}
\usepackage{graphicx}
\usepackage{comment}
\usepackage{bm}

\newtheorem{thm}{Theorem}[section]
\newtheorem{definition}[thm]{Definition}
\newtheorem{lemma}[thm]{Lemma}
\newtheorem{cor}[thm]{Corollary}
\newtheorem{prop}[thm]{Proposition}
\newtheorem{remark}[thm]{Remark}

\newcommand{\R}{\mathbb{R}}
\newcommand{\RR}{\mathbb{R}^{2}}

\newcommand{\RRD}{\mathbb{R}^{3}}

\newcommand{\Umus}{\textbf{U}^{\mu,\gamma}_{\! \sslash}}
\newcommand{\us}{u_{\! \sslash}}
\newcommand{\Umuszero}{\textbf{U}^{\mu,0}_{\! \sslash}}

\newcommand{\lver}{\left\lvert}
\newcommand{\llver}{\left\lvert\left\lvert}
\newcommand{\rver}{\right\rvert}
\newcommand{\rrver}{\right\rvert \right\rvert}

\newcommand{\uastsh}{u^{\ast}_{\text{sh}}}
\newcommand{\vastsh}{v^{\ast}_{\text{sh}}}

\renewcommand{\footnote}[1]{\textsuperscript{\addtocounter{footnote}{1}(\thefootnote)}\footnotetext{#1}}

\def \epsilon {\varepsilon}

\begin{document}
\title{\textbf{Long wave approximation for water waves under a Coriolis forcing and the Ostrovsky equation}}
\author{Benjamin MELINAND\footnote{IMB, Universit\'e de Bordeaux. Email : benjamin.melinand@math.u-bordeaux.fr}}
\date{April 2016}

\maketitle

\vspace{-1cm}

\begin{abstract}
\noindent This paper is devoted to the study of the long wave approximation for water waves under the influence of the gravity and a Coriolis forcing. We start by deriving a generalization of the Boussinesq equations in 1D (in space) and we rigorously justify them as an asymptotic model of the water waves equations. These new Boussinesq equations are not the classical Boussinesq equations. A new term due to the vorticity and the Coriolis forcing appears that can not be neglected. Then, we study the Boussinesq regime and we derive and fully justify  different asymptotic models when the bottom is flat : a linear equation linked to the Klein-Gordon equation admitting the so-called Poincar\'e waves; the Ostrovsky equation, which is a generalization of the KdV equation in presence of a Coriolis forcing, when the rotation is weak;  and finally the KdV equation when the rotation is very weak. Therefore, this work provides the first mathematical justification of the Ostrovsky equation. Finally, we derive a generalization of the Green-Naghdi equations in 1D in space for small topography variations and we show that this model is consistent with the water waves equations.
\end{abstract}

\section{Introduction}

\noindent We study the motion of an incompressible, inviscid fluid with a constant density $\rho$ and no surface tension under the influence of the gravity $\bm{g} = -g \bm{e_{z}}$ and the rotation of the Earth with a rotation vector $\textbf{f} = \frac{f}{2} \bm{e_{z}}$.  We suppose that the seabed and the surface are graphs above the still water level. The horizontal variable is $X=(x,y) \in \RR$ and $z \in \mathbb{R}$ is the vertical variable. The water occupies the domain $\Omega_{t} := \{ (X,z) \in \RRD \text{ , } -H + b(X) < z < \zeta (t,X) \}$. The velocity in the fluid domain is denoted $\textbf{U} = \left( \textbf{V}, \text{w} \right)^{t}$ where $\textbf{V}$ is the horizontal component of $\textbf{U}$ and $\text{w}$ its vertical component. The equations governing such a fluid are the free surface Euler-Coriolis equations\footnote{We  consider that the centrifugal potential is constant and included in the pressure term.}

\begin{equation}\label{Euler_equations}
\left\{
\begin{aligned}
&\partial_{t} \textbf{U} + \left( \textbf{U} \cdot \nabla_{\! X,z} \right) \textbf{U} + \textbf{f} \times \textbf{U} = - \frac{1}{\rho} \nabla_{\! X,z} \mathcal{P} - g \bm{e_{z}} \text{ in } \Omega_{t},\\
&\text{div} \; \textbf{U} = 0 \text{ in } \Omega_{t},
\end{aligned}
\right.
\end{equation}

\noindent with the boundary conditions

\begin{equation}\label{bound_cond}
\left\{
\begin{aligned}
&\mathcal{P}_{|z=\zeta} = P_{0}, \\
&\partial_{t} \zeta	- \underline{\textbf{U}} \cdot \textbf{N} = 0, \\
&\textbf{U}_{b} \cdot \textbf{N}_{b} = 0,
\end{aligned}
\right.
\end{equation}

\noindent where $P_{0}$ is constant, $\textbf{N} = \begin{pmatrix} - \nabla \zeta \\ 1 \end{pmatrix}$, $\textbf{N}_{b} = \begin{pmatrix} - \nabla b \\ 1 \end{pmatrix}$, $\underline{\textbf{U}} = \begin{pmatrix} \underline{\textbf{V}} \\ \underline{\text{w}} \end{pmatrix} = \textbf{U}_{|z=\zeta}$ and $\textbf{U}_{b} = \begin{pmatrix} \textbf{V}_{b} \\ \text{w}_{b} \end{pmatrix} = \textbf{U}_{|z=-H+b}.$

\medskip

\noindent Influenced by the works of Zakharov (\cite{Zakharov}) and Craig-Sulem-Sulem (\cite{Craig_Sulem_2}), Castro and Lannes in \cite{Castro_Lannes_vorticity} shown that we can express the free surface Euler equations thanks to the unknowns $\left(\zeta, \textbf{U}_{\! \sslash}, \bm{\omega} \right)$\footnote{In fact, Castro and Lannes used the unknowns $\left(\zeta, \frac{\nabla}{\Delta} \cdot \textbf{U}_{\! \sslash}, \bm{\omega} \right)$. But the unknowns $\left(\zeta, \textbf{U}_{\! \sslash}, \bm{\omega} \right)$ are better to derive shallow water asymptotic models.} where 

\begin{equation*}
\textbf{U}_{\! \sslash} = \underline{\textbf{V}} + \underline{\text{w}} \nabla \zeta,
\end{equation*}

\noindent and $\bm{\omega}$ is the vorticity of the fluid. Then, they gave a system of three equations on these unknowns. In \cite{my_coriolis} we proceeded as Castro and Lannes and, taking into account the Coriolis force, we got the following system, called the Castro-Lannes system or the water waves equations,

\begin{equation}\label{castro_lannes_formulation_dim}
\left\{
\begin{aligned}
&\hspace{-0.05cm} \partial_{t} \zeta -\underline{\textbf{U}}\cdot \textbf{N} = 0,\\
&\partial_{t} \textbf{U}_{\! \sslash} \hspace{-0.1cm} +  \hspace{-0.1cm} \nabla \zeta  \hspace{-0.1cm} +  \hspace{-0.1cm} \frac{1}{2} \nabla \hspace{-0.1cm} \left\lvert \textbf{U}_{\! \sslash} \right\rvert^{2}  \hspace{-0.1cm} -  \hspace{-0.1cm} \frac{1}{2} \nabla \hspace{-0.1cm} \left[ \left( 1 +  \left\lvert \nabla \zeta\right\rvert^{2} \right) \underline{\text{w}}^{2} \right]  \hspace{-0.1cm} +  \hspace{-0.1cm} \underline{\bm{\omega}} \hspace{-0.05cm} \cdot \hspace{-0.05cm} \textbf{N} \hspace{0.05cm} \underline{\textbf{V}}^{\perp}  \hspace{-0.1cm} +  \hspace{-0.1cm} f \underline{\textbf{V}}^{\perp} = 0,\\
&\hspace{-0.05cm} \partial_{t} \bm{\omega} \hspace{-0.05cm} + \hspace{-0.05cm}  \left( \textbf{U} \hspace{-0.05cm} \cdot \hspace{-0.05cm} \nabla_{\! X,z} \hspace{-0.05cm} \right) \bm{\omega} \hspace{-0.05cm} = \left( \bm{\omega} \cdot \nabla_{\! X,z} \right) \textbf{U} \hspace{-0.05cm} + \hspace{-0.05cm} f \partial_{z} \textbf{U},
\end{aligned}
\right.
\end{equation}

\noindent where $\underline{\bm{\omega}} = \bm{\omega}_{|z=\zeta}$ and $\textbf{U} = \begin{pmatrix} \textbf{V} \\ \text{w} \end{pmatrix} = \textbf{U}[\zeta, b](\textbf{U}_{\! \sslash},\bm{\omega})$ is the unique solution in $H^{1}(\Omega_{t})$ of

\begin{equation}\label{div_curl_formulation_dim}
\left\{
\begin{aligned}
&\text{curl} \; \textbf{U} = \bm{\omega} \text{ in } \Omega_{t},\\
&\text{div} \; \textbf{U} = 0 \text{ in } \Omega_{t},\\
&\left(\underline{\textbf{V}} + \underline{\text{w}} \nabla \zeta \right)_{|z= \zeta} = \textbf{U}_{\! \sslash},\\
&\textbf{U}_{b} \cdot \textbf{N}_{b} = 0,
\end{aligned}
\right.
\end{equation}

\noindent and with the following constraint

\begin{equation}\label{constraint_vorticity_dim}
\nabla^{\perp} \cdot \textbf{U}_{\! \sslash} = \underline{\bm{\omega}} \cdot \textbf{N}. 
\end{equation}

\noindent Our principal motivation is the study of the long waves or Boussinesq regime. Hence, we nondimensionalize the previous equations. We have six physical parameters in our problem : the typical amplitude of the surface $a$, the typical amplitude of the bathymetry $a_{\text{bott}}$, the typical longitudinal scale $L_{x}$,  the typical transverse scale $L_{y}$, the characteristic water depth $H$ and the typical Coriolis frequency $f$. Then we can introduce five dimensionless parameters 

\begin{equation*}
\epsilon =\frac{a}{H} \text{, } \beta = \frac{a_{\text{bott}}}{H} \text{, } \mu = \frac{H^{2}}{L_{x}^{2}} \text{, } \gamma = \frac{L_{x}}{L_{y}} \text{ and } \text{Ro} = \frac{a \sqrt{gH}}{H f L_{x}}.
\end{equation*}

\noindent The parameter $\epsilon$ is called the nonlinearity parameter, $\beta$ is called the bathymetric parameter, $\mu$ is called the shallowness parameter, $\gamma$ is called the transversality parameter and $\text{Ro}$ is the Rossby number. Then, we can nondimensionalize the Euler equations \eqref{Euler_equations} and the Castro-Lannes equations \eqref{castro_lannes_formulation_dim} (see Part \ref{nondim}).

\medskip

\noindent We organize our paper in four parts. In Subsection \ref{nondim}, we nondimensionalize the Castro-Lannes equations (see System \eqref{castro_lannes_formulation}) and we give in Subsection \ref{useful_results} a local wellposedness result on these equations by taking into account the dependence on the dimensionless parameters. Section \ref{derive_boussi} is devoted to derive a generalization of the Boussinesq equations in 1D under a Coriolis forcing and to fully justify it. The Boussinesq equations are obtained under the assumption that $\mu$ is small, $\epsilon, \beta = \mathcal{O}(\mu)$ (Boussinesq regime) and by neglecting all the terms of order $\mathcal{O}(\mu^{2})$ in the nondimensionalized Euler equations or the water waves equations (see for instance \cite{Alvarez_Lannes} in the irrotational framework). It is a system of two equations on the free surface $\zeta$ and the vertical average of the horizontal component of the velocity denoted $\overline{\textbf{V}} = (\overline{u}, \overline{v})^{t}$ (defined in \eqref{V_average}). Our Boussinesq-Coriolis equations are a system of three equations on the surface $\zeta$, the average vertical velocity $\overline{\textbf{V}}$ and the quantity $\textbf{V}^{\sharp} = (u^{\sharp},v^{\sharp})^{t}$ (defined in \eqref{u_diese}) which is introduced to catch interactions between the vorticity and the averaged velocity. These equations are the following system

\begin{equation*}
\left\{
\begin{aligned}
&\partial_{t} \zeta + \partial_{x} \left( [1+\epsilon \zeta - \beta b] \overline{u}  \right) = 0,\\
&\left(1- \frac{\mu}{3} \partial_{x}^{2} \right) \partial_{t} \overline{u} + \partial_{x} \zeta + \epsilon \overline{u} \partial_{x} \overline{u} - \frac{\epsilon}{\text{Ro}} \overline{v} + \frac{\epsilon}{\text{Ro}} \mu^{\frac{3}{2}} \frac{1}{24} \partial^{2}_{x} \frac{v^{\sharp}}{h} = 0,\\
&\partial_{t} \overline{v} + \epsilon \overline{u} \partial_{x} \overline{v} + \frac{\epsilon}{\text{Ro}} \overline{u} = 0,\\
&\partial_{t} \frac{\textbf{V}^{\sharp}}{h} + \epsilon \overline{u} \partial_{x} \frac{\textbf{V}^{\sharp}}{h} + \frac{\epsilon}{\text{Ro}} \frac{\textbf{V}^{\sharp}}{h}^{\perp} = 0,
\end{aligned}
\right.
\end{equation*}

\noindent where $h=1+\epsilon \zeta - \beta b$. Then, in Section \ref{asympto_model} we derive and fully justify different asymptotic models in the Boussinesq regime when the bottom is flat. We first derive in Subsection \ref{model_poincare} a linear system (System \eqref{boussines_linear}) linked to the Klein-Gordon equation admitting the so-called Poincar\'e waves. Then, in Subsection \ref{model_ostrov} we study the Ostrovsky equation

\begin{equation*}
\partial_{\xi} \left( \partial_{\tau} k +  \frac{3}{2} k \partial_{\xi} k + \frac{1}{6} \partial_{\xi}^{3} k \right) = \frac{1}{2} k.
\end{equation*}

\noindent This equation, derived by Ostrovsky (\cite{ostrovsky_original}), is a generalization of the KdV equation in presence of a Coriolis forcing. We offer a rigorous justification of the Ostrovsky approximation under a weak Coriolis forcing, i.e $\frac{\epsilon}{\text{Ro}} = \mathcal{O}(\sqrt{\mu})$. Notice that this work provides the first mathematical justification of the Ostrovsky equation. In Subsection \ref{model_kdv} we fully justify the KdV approximation (equation \eqref{kdv_eq}) when the rotation is very weak, i.e when  $\frac{\epsilon}{\text{Ro}} = \mathcal{O}(\mu)$. Finally, in Section \ref{derive_gn} we derive a generalization of the Green-Naghdi equations \eqref{green_naghdi} in 1D under a Coriolis forcing with small bottom variations and we show that this system is consistent with the water waves equations. The Green-Naghdi equations are originally obtained in the irrotational framework under the assumption that $\mu$ is small and by neglecting all the terms of order $\mathcal{O}(\mu^{2})$ in the nondimensionalized Euler equations or the water waves equations (see for instance \cite{green_naghdi_uneven_bott} or Part 5.1.1.2 in \cite{Lannes_ww} for a derivation in the irrotational framework). These equations were generalized in \cite{Castro_Lannes_shallow_water} in the rotational setting but without a Coriolis forcing. We add one in the paper.

\subsection{Notations}

\noindent - If $\textbf{A} \in \R^{3}$, we denote by $\textbf{A}_{h}$ its horizontal component.

\medskip

\noindent - If $\textbf{V}  = \begin{pmatrix} u \\ v \end{pmatrix} \in \RR$, we define the orthogonal of $\textbf{V}$ by $\textbf{V}^{\perp} = \begin{pmatrix} -v \\ u \end{pmatrix}$.

\medskip

\noindent - In this paper, $C \left( \cdot \right)$ is a nondecreasing and positive function whose exact value has no importance.

\medskip

\noindent - Consider a vector field $\textbf{A}$ or a function $\text{w}$ defined on $\Omega$. Then, we denote $\underline{\textbf{A}} = \textbf{A}_{|z=\epsilon\zeta}$, $\underline{\text{w}} = \text{w}_{|z=\epsilon \zeta}$ and $\textbf{A}_{b} = \textbf{A}_{|z=-1+\beta b}$, $\text{w}_{b} = \text{w}_{|z=-1+ \beta b}$.

\medskip

\noindent - If $s \in \R$ and $f$ is a function on $\RR$, $\left\lvert f \right\rvert_{H^{s}}$ is its $H^{s}$-norm, $\left\lvert f \right\rvert_{2}$ is its $L^{2}$-norm and $\lver f \rver_{L^{\infty}}$ its  $L^{\infty}(\RR)$-norm.

\medskip

\noindent - The operator $\left( \, , \, \right)_{2}$ is the $L^{2}$-scalar product in $\RR$.

\medskip

\noindent - If $f$ is a function defined on $\RR$, we denote $\nabla f$ the gradient of $f$.

\medskip

\noindent - If $\text{w}$ is a function defined on $\Omega$, $\nabla_{\! X,z} \text{w}$ is the gradient of $\text{w}$ and $\nabla_{\! X} \text{w}$ its horizontal component.

\medskip

\noindent - If $u = u(X,z)$ is defined in $\Omega$, we define

\begin{equation*}
\overline{u}(X) = \frac{1}{1+\epsilon \zeta - \beta b} \int_{-1+\beta b(X)}^{\epsilon \zeta(X)} u(X,z) dz \text{   and   } u^{\ast} = u - \overline{u}. 
\end{equation*}

\subsection{Nondimensionalization and the Castro-Lannes formulation}\label{nondim}

\noindent We recall the five dimensionless parameter

\begin{equation}\label{parameters}
\epsilon =\frac{a}{H} \text{, } \beta = \frac{a_{\text{bott}}}{H} \text{, } \mu = \frac{H^{2}}{L_{x}^{2}} \text{, } \gamma = \frac{L_{x}}{L_{y}} \text{ and } \text{Ro} = \frac{a \sqrt{gH}}{H f L_{x}}.
\end{equation}

\noindent  We nondimensionalize the variables and the unknowns. We introduce (see \cite{Lannes_ww} or \cite{my_coriolis} for instance for an explanation of this nondimensionalization)

\begin{equation}\label{nondimens_variables}
 \left\{ 
 \begin{aligned}
 &x' = \frac{x}{L_{x}} \text{, } y' = \frac{y}{L_{y}} \text{, } z' = \frac{z}{H} \text{, } \zeta' = \frac{\zeta}{a} \text{, } b' = \frac{b}{a_{\text{bott}}} \text{, } t' = \frac{\sqrt{gH}}{L_{x}} t \text{, }\\
 &\textbf{V}' = \sqrt{\frac{H}{g}} \frac{\textbf{V}}{a} \text{, } \text{w}'= H \sqrt{\frac{H}{g}} \frac{\text{w}}{aL_{x}} \text{ and } \mathcal{P}' = \frac{\mathcal{P}}{\rho g H}.
 \end{aligned}
 \right.
\end{equation}

\noindent In this paper, we use the following notations

\begin{equation}\label{nondim_op}
\nabla^{\gamma} = \nabla^{\gamma}_{\! X'} = \begin{pmatrix} \partial_{x'} \\ \gamma \partial_{y'} \end{pmatrix} \text{, } = \nabla^{\mu, \gamma}_{\! X',z'} = \begin{pmatrix} \sqrt{\mu} \nabla^{\gamma}_{\! X'} \\ \partial_{z'} \end{pmatrix} \text{, } \text{curl}^{\mu, \gamma} = \nabla^{\mu, \gamma}_{\! X',z'} \times \text{, } \text{div}^{\mu, \gamma} = \nabla^{\mu, \gamma}_{\! X',z'} \cdot .
\end{equation}

\noindent We also define 

\begin{equation}\label{nondim_U}
\textbf{U}^{\mu} = \begin{pmatrix} \sqrt{\mu} \textbf{V}' \\ \text{w}' \end{pmatrix} \text{, } \bm{\omega}'=\frac{1}{\mu} \text{curl}^{\mu, \gamma} \textbf{U}^{\mu} \text{, } \underline{\textbf{U}}^{\mu} = \begin{pmatrix} \sqrt{\mu} \underline{\textbf{V}}' \\ \underline{\text{w}}' \end{pmatrix} = \textbf{U}^{\mu}_{|z'= \epsilon \zeta'} \text{, } \textbf{U}^{\mu}_{b} = \textbf{U}^{\mu}_{|z'=-1+ \beta b'},
\end{equation}

\noindent and 

\begin{equation}
 \textbf{N}^{\mu, \gamma} =  \begin{pmatrix} - \epsilon \sqrt{\mu} \nabla^{\gamma} \zeta' \\ 1 \end{pmatrix} \text{, } \textbf{N}_{\! b}^{\mu, \gamma} = \begin{pmatrix} - \beta \sqrt{\mu} \nabla^{\gamma} b' \\ 1 \end{pmatrix}.
\end{equation}

\noindent Notice that our nondimensionalization of the vorticity allows us to consider only weakly sheared flows (see \cite{Castro_Lannes_shallow_water}, \cite{teshukov_shear_flows}, \cite{richard_gravi_shear_flows}). The nondimensionalized fluid domain is

\begin{equation}\label{Omega_t}
\Omega^{\prime}_{t'} := \{ (X',z') \in \RRD \text{ , } -1 + \beta b'(X') < z' < \epsilon \zeta' (t',X') \}.
\end{equation}

\noindent Finally, if $\textbf{V}  = \begin{pmatrix} u \\ v \end{pmatrix} \in \RR$, we define $\textbf{V}$ by $\textbf{V}^{\perp} = \begin{pmatrix} -v \\ u \end{pmatrix}$. Then, the Euler-Coriolis equations \eqref{Euler_equations} become 

\begin{equation}\label{nondim_Euler_equations}
\left\{
\begin{aligned}
&\partial_{t'} \textbf{U}^{\mu} + \frac{\epsilon}{\mu} \left( \textbf{U}^{\mu} \cdot \nabla_{\! X',z'}^{\mu, \gamma} \right) \textbf{U}^{\mu} + \frac{\epsilon \sqrt{\mu}}{\text{Ro}} \begin{pmatrix} \;\; \textbf{V}^{\prime \perp} \\ 0 \end{pmatrix} = - \frac{1}{\epsilon} \nabla_{\! X',z'}^{\mu, \gamma} \mathcal{P}' - \frac{1}{\epsilon} \bm{e_{z}} \text{ in } \Omega^{\prime}_{t},\\
&\text{div}^{\mu, \gamma}_{\! X' \! ,z'} \; \textbf{U}^{\mu} = 0 \text{ in } \Omega^{\prime}_{t},
\end{aligned}
\right.
\end{equation}

\noindent with the boundary conditions

\begin{equation}\label{nondim_bound_cond}
\left\{
\begin{aligned}
&\partial_{t'} \zeta'	- \frac{1}{\mu} \underline{\textbf{U}}^{\mu} \cdot \textbf{N}^{\mu, \gamma} = 0, \\
&\textbf{U}^{\mu}_{b} \cdot \textbf{N}_{b}^{\gamma, \mu} = 0.
\end{aligned}
\right.
\end{equation}

\noindent We can also nondimensionalize the Castro-Lannes formulation. We introduce the quantity

\begin{equation*}
\Umus = \underline{\textbf{V}} + \epsilon \underline{\text{w}} \nabla^{\gamma} \zeta.
\end{equation*}

\noindent Then, the Castro-Lannes formulation becomes (see \cite{Castro_Lannes_vorticity} or \cite{my_coriolis} when $\gamma = 1$),

\begin{equation}\label{castro_lannes_formulation}
\left\{
\begin{aligned}
&\hspace{-0.05cm} \partial_{t} \zeta - \frac{1}{\mu} \underline{\textbf{U}}^{\mu} \cdot \textbf{N}^{\mu, \gamma} = 0,\\
&\partial_{t} \Umus \hspace{-0.1cm} +  \hspace{-0.1cm} \nabla^{\gamma} \zeta  \hspace{-0.1cm} +  \hspace{-0.1cm} \frac{\epsilon}{2} \nabla^{\gamma}  \hspace{-0.1cm} \left\lvert \Umus \right\rvert^{2}  \hspace{-0.1cm} -  \hspace{-0.1cm} \frac{\epsilon}{2 \mu} \nabla^{\gamma}  \hspace{-0.1cm} \left[ \left( 1 +  \epsilon^{2} \mu \left\lvert \nabla^{\gamma} \zeta\right\rvert^{2} \right) \underline{\text{w}}^{2} \right]  \hspace{-0.1cm} +  \hspace{-0.1cm} \epsilon \underline{\bm{\omega}} \hspace{-0.05cm} \cdot \hspace{-0.05cm} \textbf{N}^{\mu, \gamma} \hspace{0.05cm} \underline{\textbf{V}}^{\perp}  \hspace{-0.1cm} +  \hspace{-0.1cm} \frac{\epsilon}{\text{Ro}} \underline{\textbf{V}}^{\perp} = 0,\\
&\hspace{-0.05cm} \partial_{t} \bm{\omega} \hspace{-0.05cm} + \hspace{-0.05cm} \frac{\epsilon}{\mu} \hspace{-0.05cm} \left( \hspace{-0.05cm} \textbf{U}^{\mu} \hspace{-0.05cm} \cdot \hspace{-0.05cm} \nabla^{\mu, \gamma}_{\! X,z} \hspace{-0.05cm} \right) \bm{\omega} \hspace{-0.05cm} = \hspace{-0.05cm} \frac{\epsilon}{\mu} \left( \bm{\omega} \cdot \nabla^{\mu, \gamma}_{\! X,z} \right) \textbf{U}^{\mu} \hspace{-0.05cm} + \hspace{-0.05cm} \frac{\epsilon}{\mu \text{Ro}} \partial_{z} \textbf{U}^{\mu},
\end{aligned}
\right.
\end{equation}

\noindent where $\textbf{U}^{\mu} = \begin{pmatrix} \sqrt{\mu} \textbf{V} \\ \text{w} \end{pmatrix} = \textbf{U}^{\mu}[\epsilon \zeta, \beta b](\Umus,\bm{\omega})$ is the unique solution in $H^{1}(\Omega_{t})$ of

\begin{equation}\label{div_curl_formulation}
\left\{
\begin{aligned}
&\text{curl}^{\mu, \gamma} \; \textbf{U}^{\mu} = \mu \bm{\omega} \text{ in } \Omega_{t},\\
&\text{div}^{\mu, \gamma} \; \textbf{U}^{\mu} = 0 \text{ in } \Omega_{t},\\
&\left(\underline{\textbf{V}} + \epsilon \underline{\text{w}} \nabla^{\gamma} \zeta \right)_{|z=\epsilon \zeta} = \Umus,\\
&\textbf{U}_{b}^{\mu} \cdot \textbf{N}^{\mu, \gamma}_{b} = 0,
\end{aligned}
\right.
\end{equation}

\noindent and with the following constraint

\begin{equation}\label{constraint_vorticity}
\nabla^{\perp} \cdot \Umus = \underline{\bm{\omega}} \cdot \textbf{N}^{\mu, \gamma}. 
\end{equation}

\begin{remark}\label{gamma_equal_zero}
\noindent When, $\bm{\omega}=0$ and $\text{Ro}=+\infty$, we get the irrotational water waves equations (see Remark 2.4 in \cite{Castro_Lannes_vorticity}). In particular in this situation, when $\gamma = 0$ we can check that the velocity \upshape$\textbf{U}^{\mu}$ \itshape becomes two dimensional : \upshape$\textbf{U}^{\mu} = \left( \sqrt{\mu} \textbf{V}_{x} , 0 , \text{w} \right)^{t}$\itshape. This is not the case when $\bm{\omega} \neq 0$. Even if $\gamma = 0$, the vorticity transfers energy from \upshape $\textbf{V}_{x}$ \itshape to \upshape $\textbf{V}_{y}$\itshape. The only way to get a two dimensional speed is to assume that $\omega = \left( 0 , \omega_{y} , 0 \right)^{t}$ (see for instance \cite{lannes_marche}).
\end{remark}

\begin{remark}\label{contrainst_vorti_initially}
\noindent Notice that if \upshape $\left(\zeta, \Umus, \bm{\omega} \right)$ \itshape is a solution of the Castro-Lannes system \eqref{castro_lannes_formulation}, \upshape $\nabla^{\perp} \cdot \Umus$ \itshape satisfies the equation

\upshape
\begin{equation*}
\partial_{t} \nabla^{\perp} \cdot \Umus + \nabla^{\gamma} \cdot \left(\epsilon \underline{\bm{\omega}}  \cdot  \textbf{N}^{\mu, \gamma} \underline{\textbf{V}}^{\perp}   + \frac{\epsilon}{\text{Ro}} \underline{\textbf{V}} \right) = 0.
\end{equation*}
\itshape

\noindent Furthermore, by taking the trace of the third equation of the Castro-Lannes system \eqref{castro_lannes_formulation}, we can see that \upshape $\underline{\bm{\omega}}  \cdot  \textbf{N}^{\mu, \gamma}$ \itshape satisfies the equation

\upshape
\begin{equation*}
\partial_{t} \left( \underline{\bm{\omega}}  \cdot  \textbf{N}^{\mu, \gamma} \right) + \nabla^{\gamma} \cdot \left(\epsilon \underline{\bm{\omega}}  \cdot  \textbf{N}^{\mu, \gamma} \underline{\textbf{V}}^{\perp}   + \frac{\epsilon}{\text{Ro}} \underline{\textbf{V}} \right) = 0,
\end{equation*}
\itshape

\noindent Hence, the constraint \eqref{constraint_vorticity} is propagated by the equations.
\end{remark}

\noindent We add a technical assumption. We assume that the water depth is bounded from below by a positive constant

\begin{equation}\label{nonvanishing}
  \exists \, h_{\min} > 0 \text{ ,  } 1 + \epsilon \zeta - \beta b \geq h_{\min}.
\end{equation} 

\noindent We also suppose that the dimensionless parameters satisfy

\begin{equation}\label{constraints_parameters}
\exists \mu_{\max} \text{, } 0 < \mu \leq \mu_{\max} \text{, } 0 < \epsilon \leq 1 \text{, } 0 \leq \gamma \leq 1 \text{, } 0 \leq \beta \leq 1 \text{ and } \frac{\epsilon}{\text{Ro}} \leq 1.
\end{equation}

\begin{remark}\label{condition_coriolis}
\noindent We have $\frac{\epsilon}{\text{Ro}} = \frac{fL}{\sqrt{gH}}$. As said in \cite{my_coriolis}, it is quite reasonable to assume that $\frac{\epsilon}{\text{Ro}} \leq 1$ since for water waves, the typical rotation speed due to the Coriolis forcing is less than the typical water wave celerity (see for instance \cite{pedlosky}, \cite{gill}, \cite{leblond}).
\end{remark}

\subsection{Useful results}\label{useful_results}

\noindent In this paper, we fully justify different asymptotic models of the water waves equations. Then, we have to define the notion of consistence (see for instance \cite{Lannes_ww}). 

\begin{definition}\label{consistence}
The Castro-Lannes equations \eqref{castro_lannes_formulation} are consistent of order $\mathcal{O} \left( \mu^{k} \right)$ with a system of equations $S$ for $\zeta$ and \upshape $\overline{\textbf{V}}$ \itshape if for all sufficiently smooth solutions \upshape $\left(\zeta, \Umus, \bm{\omega} \right)$ \itshape of the Castro-Lannes equations \eqref{castro_lannes_formulation} , the pair \upshape $\left(\zeta, \overline{\textbf{V}}[\epsilon \zeta, \beta b] \left(\Umus, \bm{\omega} \right) \right)$ (defined in \eqref{V_average}) \itshape solves $S$ up to a residual of order $\mathcal{O} \left( \mu^{k} \right)$. 
\end{definition}

\noindent We also need an existence result for the Castro-Lannes formulation \eqref{castro_lannes_formulation}. This is the purpose of the next theorem proven in \cite{my_coriolis}. We recall that the existence of the water waves equations is always under the so-called Rayleigh-Taylor condition assuming the positivity of the Rayleigh-Taylor coefficient $\mathfrak{a}$ (see Part 3.4.5 in \cite{Lannes_ww} for the link between $\mathfrak{a}$ and the Rayleigh-Taylor condition or \cite{my_coriolis})  where

\begin{equation}\label{rayleigh_taylor_coefficient}
\mathfrak{a}: = \mathfrak{a}[\epsilon \zeta, \beta b](\Umus, \bm{\omega}) = 1 + \epsilon \left( \partial_{t} + \epsilon \underline{\textbf{V}}[\epsilon \zeta, \beta b](\Umus, \bm{\omega}) \cdot \nabla \right) \underline{\text{w}}[\epsilon \zeta, \beta b](\Umus, \bm{\omega}).
\end{equation}

\noindent Notice that in \cite{my_coriolis} we explain how we can define initially the Rayleigh-Taylor coefficient $\mathfrak{a}$.

\begin{thm}\label{existence_ww}
\noindent Let \upshape$A > 0$, $\textbf{N} \geq 5$, $b \in H^{N+2}(\RR)$. We assume that

\upshape
\begin{equation*}
\left( \zeta_{0}, (\Umus)_{0}, \bm{\omega}_{0} \right) \in H^{N}(\RR) \times H^{N}(\RR) \times H^{N-1}(\Omega_{0}),
\end{equation*}
\itshape

\noindent that \upshape$\nabla^{\mu, \gamma} \cdot \bm{\omega}_{0} =0$ \itshape and that Condition \eqref{constraint_vorticity} is satisfied. We suppose that \upshape$\left( \epsilon,\beta,\gamma,\mu,\text{Ro} \right)$ \itshape satisfy \eqref{constraints_parameters}. Finally, we assume that

\upshape
\begin{equation*}
\exists \, h_{\min} \text{, } \mathfrak{a}_{\min} > 0 \text{ ,  } \epsilon \zeta _{0}+ 1 - \beta b \geq h_{\min} \text{  and  } \mathfrak{a}[\epsilon \zeta_{0}, \beta b]((\Umus)_{0}, \bm{\omega}_{0}) \geq \mathfrak{a}_{\min},
\end{equation*}
\itshape

\noindent and

\upshape
\begin{equation*}
\lver \zeta_{0} \rver_{H^{N}} + \lver (\Umus)_{0} \rver_{H^{N}} + \llver \bm{\omega}_{0} \rrver_{H^{N-1}} \leq A.
\end{equation*}
\itshape
 
\noindent Then, there exists \upshape$T > 0$ \itshape and a unique classical solution \upshape$\left( \zeta, \Umus, \bm{\omega} \right) $ \itshape to the Castro-Lannes \eqref{castro_lannes_formulation} with initial data \upshape$\left( \zeta_{0}, (\Umus)_{0}, \bm{\omega}_{0} \right)$\itshape. Moreover, 

\upshape
\begin{equation*}
T =  \frac{T_{0}}{\max(\epsilon, \beta, \frac{\epsilon}{\text{Ro}})} \text{ , }  \frac{1}{T_{0}} = c^{1} \text{ and  } \underset{[0,T]}{\max} \lver \zeta(t) \rver_{H^{N}} + \lver \Umus(t) \rver_{H^{N}} + \llver \bm{\omega}(t) \rrver_{H^{N-1}}  = c^{2},
\end{equation*}
\itshape

\noindent with \upshape $c^{j} = C \left(A, \mu_{\max}, \frac{1}{h_{\min}}, \frac{1}{\mathfrak{a}_{\min}}, \left\lvert b \right\rvert_{H^{N+2}} \right)$.\itshape
\end{thm}

\noindent Thanks to this theorem, we know that the quantities $\zeta$, $\Umus$, $\omega$ and then $\overline{\textbf{V}}$ (defined in \eqref{V_average}) remain bounded uniformly with respect to the small parameters during the time evolution of the flow, which will be essential to derive rigorously asymptotic models.

\section{Boussinesq-Coriolis equations when $\gamma = 0$}\label{derive_boussi}

\noindent This part is devoted to the derivation and the full justification of the Boussinesq-Coriolis equations \eqref{boussinesq_eq} under a Coriolis forcing and with $\gamma = 0$. These equations are an order $\mathcal{O}(\mu^{2})$ approximation of the water waves equations under the assumption that $\epsilon, \beta = \mathcal{O}(\mu)$. The corresponding regime is called \textit{long wave regime} or Boussinesq regime. Contrary to \cite{Castro_Lannes_shallow_water}, whose approach is based on the averaged Euler equations, our derivation is based on the Castro-Lannes equations \eqref{castro_lannes_formulation}. Then, the asymptotic regime is

\begin{equation}\label{bouss_regime}
\mathcal{A}_{\text{Bouss}} = \left\{ \left( \epsilon, \beta, \gamma, \mu, \text{Ro} \right), 0 \leq \mu \leq \mu_{0},  \frac{\epsilon}{\text{Ro}} \leq 1, \epsilon =  \mathcal{O}(\mu), \beta = \mathcal{O} \left( \mu \right), \gamma = 0 \right\}.
\end{equation}

\noindent 

\begin{remark}
\noindent In fact, we can relax the assumption $\gamma = 0$ by only assuming that $\gamma = \mathcal{O} \left( \mu^{2} \right)$ since we neglect all the terms of order $\mathcal{O}(\mu^{2})$ in the following.
\end{remark}

\noindent We introduce the water depth

\begin{equation}
h(t,X) = 1 + \epsilon \zeta(t,X) - \beta b(X),
\end{equation}

\noindent and the averaged horizontal velocity

\begin{equation}\label{V_average}
\overline{\textbf{V}} = \overline{\textbf{V}}[\epsilon \zeta, \beta b](\Umus,\bm{\omega})(t,X) = \frac{1}{h(t,X)} \int_{z=-1+\beta b(X)}^{\epsilon \zeta(t,X)} \textbf{V}[\epsilon \zeta, \beta b](\Umus,\bm{\omega})(t,X,z) dz.
\end{equation}

\noindent More generally, if $u$ is a function defined in $\Omega$, $\overline{u}$ is its average and $u^{\ast} = u - \overline{u}$. In the following we denote $\textbf{V} = \left(u, v \right)^{t}$. As noticed in \cite{Castro_Lannes_vorticity}, we have to introduce the "shear" velocity

\begin{equation}\label{V_shear}
\textbf{V}_{\text{sh}} = \textbf{V}_{\text{sh}}[\epsilon \zeta, \beta b](\Umus,\bm{\omega})(t,X) = \left(u_{\text{sh}}, v_{\text{sh}} \right) = \int_{z}^{\epsilon \zeta } \bm{\omega}_{h}^{\perp} 
\end{equation}

\noindent and its average

\begin{equation*}
\textbf{Q} = \left(\text{Q}_{x}, \text{Q}_{y} \right)^{t} = \overline{\textbf{V}}_{\text{sh}} = \frac{1}{h} \int_{-1+\beta b}^{\epsilon \zeta} \int_{z'}^{\epsilon \zeta} \bm{\omega}_{h}^{\perp}.
\end{equation*}

\noindent When $\gamma = 0$, $\Umus = \left(\underline{u} + \epsilon \underline{\text{w}} \partial_{x} \zeta, \underline{v} \right)^{t}$. Hence in the following, we denote 

\begin{equation}\label{def_us}
\us = \underline{u} + \epsilon \underline{\text{w}} \partial_{x} \zeta.
\end{equation}

\noindent In this section, we do the asymptotic expansion with respect to $\mu$ of different quantities.  In the following, we denote by $R$ a remainder whose exact value has no importance and which is bounded uniformly with respect to $\mu$.

\begin{remark}\label{control_remainder}
\noindent Notice that thanks to Theorem \ref{existence_ww}, we know that the quantities \upshape $\zeta$, $\Umus$, $\omega$ , $\overline{\textbf{V}}$ and $\textbf{U}$ \itshape remain bounded uniformly with respect to the small parameters during the time evolution of the flow. Furthermore, \upshape $\partial_{t} \zeta$, $\partial_{t} \Umus$, $\partial_{t} \omega$ and $\partial_{t} \textbf{U}$ \itshape also remain bounded uniformly with respect to the small parameters during this time.
\end{remark}

\subsection{Asymptotic expansion for the velocity and useful identities}

\noindent In this part, we give an expansion of the velocity with respect to $\mu$. First we recall the following fact (the proof is a small adaptation of Proposition 4.2 in \cite{my_coriolis}).

\begin{prop}\label{mean_eq}
\noindent If \upshape $\left(\zeta, \Umus, \bm{\omega} \right)$ \itshape satisfy the Castro-Lannes system \eqref{castro_lannes_formulation}, we have

\upshape
\begin{equation*}
\underline{\textbf{U}}^{\mu} \cdot \textbf{N}^{\mu, \gamma} = - \mu \nabla^{\gamma} \cdot \left(h \overline{\textbf{V}} \right).
\end{equation*}
\itshape

\end{prop}

\noindent This proposition, coupled with the first equation of \eqref{castro_lannes_formulation}, gives us an equation that links $\zeta$ to $\overline{\textbf{V}}$. In particular, when $\gamma = 0$, we get an equation that links $\zeta$ to $\overline{u}$. We also need an expansion of $u$ and $v$ with respect to $\mu$. The following proposition is for $v$.

\begin{prop}\label{equation_v}
\noindent If \upshape $\left(\zeta, \Umuszero, \bm{\omega} \right)$ \itshape satisfy the Castro-Lannes system \eqref{castro_lannes_formulation}, we have

\upshape
\begin{equation*}
\begin{aligned}
&v = \overline{v} + \sqrt{\mu} \vastsh,\\
&\underline{v} = \overline{v} - \sqrt{\mu} \text{Q}_{y},\\
&\underline{\bm{\omega}} \cdot \textbf{N}^{\mu,0} = \partial_{x} \underline{v}
\end{aligned}
\end{equation*}
\itshape

\noindent and 

\upshape
\begin{equation*}
\partial_{t} \underline{v} + \epsilon \underline{u} \partial_{x} \underline{v} + \frac{\epsilon}{\text{Ro}} \underline{u} = 0.
\end{equation*}
\itshape

\end{prop}

\begin{proof}

\noindent Since $\text{curl}^{\mu, 0} \; \textbf{U}^{\mu} = \mu \bm{\omega}$, we get that

\begin{equation}\label{interm}
\sqrt{\mu} \bm{\omega}_{x} = - \partial_{z} v \text{   and   } \bm{\omega}_{z} = \partial_{x} v.
\end{equation}

\noindent Then, plugging the ansatz $v = \overline{v} + \sqrt{\mu} v_{1}$ in the first equation and using the fact that the average of $v_{1}$ is equal to $0$ we get

\begin{equation*}
\underline{v} = \overline{v} - \sqrt{\mu} \frac{1}{h} \int_{-1+\beta b}^{\epsilon \zeta} \int_{z'}^{\epsilon \zeta} \bm{\omega}_{x}.
\end{equation*}

\noindent Furthermore, from the equation on the second component of $\Umuszero$, we have

\begin{equation*}
\partial_{t} \underline{v} + \epsilon \bm{\omega} \cdot \textbf{N}^{\mu,0} \underline{u} + \frac{\epsilon}{\text{Ro}} \underline{u} = 0.
\end{equation*}

\noindent Then, using the second equation of \eqref{interm}, we get that $\underline{\bm{\omega}} \cdot \textbf{N}^{\mu,0} = \partial_{x} \underline{v}$ and the result follows.

\end{proof}

\noindent The expansion of $u$ is more complex and also involves an expansion of $\text{w}$. It is the purpose of the following proposition. But before, we also have to introduce the following operators

\begin{equation*}
T \left[\epsilon \zeta, \beta b \right] f = \int_{z}^{\epsilon \zeta} \partial^{2}_{x} \int_{-1+\beta b}^{z'} f \text{     and     } T^{\ast} \left[\epsilon \zeta, \beta b \right] f = \left(  T \left[\epsilon \zeta, \beta b \right] f \right)^{\ast},
\end{equation*}

\noindent When no confusion is possible, we denote $T = T \left[\epsilon \zeta, \beta b \right]$ and $T^{\ast} = T^{\ast} \left[\epsilon \zeta, \beta b \right]$.

\begin{prop}\label{equation_u}
\noindent If \upshape $\left(\zeta, \Umuszero, \bm{\omega} \right)$ \itshape satisfy the Castro-Lannes system \eqref{castro_lannes_formulation}, we have

\upshape
\begin{equation*}
\begin{aligned}
&u = \overline{u} + \sqrt{\mu} u_{\text{sh}}^{\ast} + \mu T^{\ast} \overline{u} + \mu^{\frac{3}{2}} T^{\ast} u_{\text{sh}}^{\ast} + \mu^{2} R,\\
&\underline{u} = \overline{u} - \sqrt{\mu} \text{Q}_{x} + \mu \underline{T^{\ast} \overline{u}} - \mu^{\frac{3}{2}}  \overline{T u_{\text{sh}}^{\ast}} + \mu^{2} R,
\end{aligned}
\end{equation*}
\itshape

\noindent where $T^{\ast} \overline{u} = - \frac{1}{2} \left( \left[z+1-\beta b \right]^{2} - \frac{h^{2}}{3} \right) \partial_{x}^{2} \overline{u} + \beta R$. We also have

\upshape
\begin{equation*}
\begin{aligned}
&\text{w} = - \mu \partial_{x} \left(\int_{-1+\beta b}^{z} u \right),\\
&\underline{\text{w}} = - \mu h \partial_{x} \overline{u} - \mu^{\frac{3}{2}} \partial_{x} h \text{Q}_{x} + \max(\mu^{2}, \beta \mu) R,
\end{aligned}
\end{equation*}
\itshape

\noindent and 

\upshape
\begin{equation*}
\us = \overline{u} - \sqrt{\mu}  \text{Q}_{x} - \mu \frac{1}{3h} \partial_{x} \left( h^{3} \partial_{x} \overline{u} \right)- \mu^{\frac{3}{2}} \left( \overline{T u^{\ast}_{\text{sh}}} + \text{Q}_{x} \left( \partial_{x} h \right)^{2} \right) + \max(\mu^{2}, \beta \mu) R.
\end{equation*}
\itshape

\end{prop}

\begin{proof}

\noindent This proof is a small adaptation of part 2.2 in \cite{Castro_Lannes_shallow_water} and Part 4.2 in  \cite{my_coriolis}. We recall the main steps. Using the fact that the velocity is divergence free and Proposition \ref{mean_eq}, we get

\begin{equation*}
\text{w} = - \mu \partial_{x} \left(\int_{-1+\beta b}^{z} u \right).
\end{equation*}

\noindent Furthermore, since $\text{curl}^{\mu, 0} \; \textbf{U}^{\mu} = \mu \bm{\omega}$, we get that

\begin{equation*}
\sqrt{\mu} \bm{\omega}_{y} = \partial_{z} u - \partial_{x} \text{w}.
\end{equation*}

\noindent Then, plugging the ansatz $u = \overline{u} + \sqrt{\mu} u_{1}$ and using the fact that the average of $u_{1}$ is zero, we get

\begin{equation*}
u_{1} = - \left(\int_{z}^{\epsilon \zeta} \bm{\omega}_{y} \right)^{\ast} - \frac{1}{\sqrt{\mu}} \left(\int_{z}^{\epsilon \zeta} \partial_{x} \text{w} \right)^{\ast}
\end{equation*}

\noindent and 

\begin{equation}\label{exact_u}
u = \overline{u} + \sqrt {\mu} u_{\text{sh}}^{\ast} + \mu T^{\ast} u.
\end{equation}

\noindent Then, the expansion for $u$ follows by applying $1+ \mu T^{\ast}$ to the previous equation. Notice that $\underline{T^{\ast} u} = - \overline{T u}$. The computation of $T^{\ast} \overline{u}$ follows from the fact that $\overline{u}$ does not depend on $z$. Finally, the expansion $\underline{\text{w}}$ and $\us$ is the direct consequence for Proposition \ref{mean_eq} and the expansion of $u$.

\end{proof}

\noindent Thanks to the previous proposition, we can also get an expansion of $\partial_{t} u$ and $\partial_{t} \text{w}$.

\begin{prop}\label{equation_partial_t_u}
\noindent If \upshape $\left(\zeta, \Umuszero, \bm{\omega} \right)$ \itshape satisfy the Castro-Lannes system \eqref{castro_lannes_formulation}, we have

\upshape
\begin{equation}
\begin{aligned}
&\partial_{t} \left(u - \overline{u} - \sqrt{\mu} u_{\text{sh}}^{\ast} - \mu T^{\ast} \overline{u} - \mu^{\frac{3}{2}} T^{\ast} u_{\text{sh}}^{\ast} \right) = \mu^{2} R,\\
&\partial_{t} \left(\underline{u} - \overline{u} + \sqrt{\mu} \text{Q}_{x} - \mu \underline{T^{\ast} \overline{u}} + \mu^{\frac{3}{2}} \overline{T u_{\text{sh}}^{\ast}} \right) = \mu^{2} R,\\
&\partial_{t} \left( \underline{\text{w}} + \mu h \partial_{x} \overline{u} + \mu^{\frac{3}{2}} \partial_{x} h \text{Q}_{x} \right) = \max(\mu^{2}, \beta \mu) R.
\end{aligned}
\end{equation}
\itshape

\end{prop}

\begin{proof}
\noindent From Equality \eqref{exact_u} we get that

\begin{equation}
u = \left(1- \mu T^{\ast} \right) \left( \overline{u} + \sqrt {\mu} u_{\text{sh}}^{\ast} \right) + \mu^{2} T^{\ast} T^{\ast} u.
\end{equation}

\noindent Hence the first and the second equations follows from Remark \ref{control_remainder}. For the third equation, we get the result thanks to Proposition \eqref{mean_eq} and Remark \ref{control_remainder}.
\end{proof}

\noindent As \cite{Castro_Lannes_shallow_water} noticed, we can not express $\overline{T u^{\ast}_{\text{sh}}}$ in terms of $\zeta$ and $\overline{\textbf{V}}$. Then, we have to introduce

\begin{equation}\label{u_diese}
\begin{aligned}
\textbf{V}^{\sharp} = (u^{\sharp}, v^{\sharp})^{t} &= - \frac{24}{h^{3}} \int_{-1+\beta b}^{\epsilon \zeta} \int_{z}^{\epsilon \zeta} \int_{-1+\beta b}^{z} \left(\uastsh,\vastsh\right)^{t},\\
&= \frac{12}{h^{3}} \int_{-1+\beta b}^{\epsilon \zeta} (1+z-\beta b)^{2} \left(\uastsh,\vastsh\right)^{t}.
\end{aligned}
\end{equation}

\noindent Notice that the previous equality follows from a double integration by parts. We have the following Lemma.

\begin{lemma}\label{udiese_T_eq}
We have the following equalities

\begin{equation*}
\begin{aligned}
\overline{T u^{\ast}_{\text{sh}}} &= - \left(\epsilon \partial_{x} \zeta \right)^{2} \text{Q}_{x} + \frac{1}{h}  \int_{-1+\beta b}^{\epsilon \zeta} \partial_{x} \int_{z}^{\epsilon \zeta} \partial_{x} \int_{-1+\beta b}^{z} u_{\text{sh}}^{\ast} \\
&= - \left(\partial_{x} h \right)^{2} \text{Q}_{x} - \frac{1}{24 h} \partial_{x}^{2} \left( h^{3} u^{\sharp} \right) + \beta R.
\end{aligned}
\end{equation*}

\end{lemma}

\begin{proof}
\noindent We have

\begin{equation}
\partial_{x} \int_{z}^{\epsilon \zeta} \partial_{x} \int_{-1+\beta b}^{z} u^{\ast}_{\text{sh}} = \int_{z}^{\epsilon \zeta} \partial_{x}^{2} \int_{-1+\beta b}^{z} u^{\ast}_{\text{sh}} + \epsilon \partial_{x} \zeta \underline{\partial_{x} \int_{-1+\beta b}^{z} u_{\text{sh}}^{\ast}}
\end{equation}

\noindent and the first equality follows from the fact that the average of $u^{\ast}_{\text{sh}}$ is zero and that $\underline{\uastsh} = -\text{Q}_{x}$. The second equality follows from the same arguments.
\end{proof}

\noindent In the following section, we give equations for $\text{Q}_{x}$, $\text{Q}_{y}$ $\textbf{V}^{\sharp}$ since we can not express these quantities with respect to $\zeta$ and $\overline{\textbf{V}}$. These equations are essential to derive the Boussinesq-Coriolis equations.

\subsection{Equations for $\text{Q}_{x}$, $\text{Q}_{y}$ and $\textbf{V}^{\sharp}$}

\noindent In this part we give the equations satisfied by $\text{Q}_{x}$ and $\text{Q}_{y}$ at order $\mathcal{O} \left( \mu^{\frac{3}{2}} \right)$. The computations are similar to Part 5.4.1 in \cite{Castro_Lannes_shallow_water}. We start by $\text{Q}_{x}$.

\begin{prop}\label{eq_Qx}
If \upshape $\left(\zeta, \Umuszero, \bm{\omega} \right)$ \itshape satisfy the Castro-Lannes system \eqref{castro_lannes_formulation}, then, in the Boussinesq regime \eqref{bouss_regime}, $\text{Q}_{x}$ satisfies the following equation

\upshape
\begin{equation*}
\partial_{t} Q_{x} + \epsilon \overline{u} \partial_{x} Q_{x} + \epsilon Q_{x} \partial_{x} \overline{u} + \frac{\epsilon}{\text{Ro} \sqrt{\mu}} \left(\underline{v} - \overline{v} \right) = \mu^{\frac{3}{2}} R,
\end{equation*}
\itshape

\noindent and $\uastsh$ satisfies the equation

\upshape
\begin{equation*}
\partial_{t} \uastsh + \epsilon \overline{u} \partial_{x} \uastsh + \epsilon \uastsh \partial_{x} \overline{u} + \frac{\epsilon}{\text{Ro} \sqrt{\mu}} \left(\overline{v} - v \right) = \mu^{\frac{3}{2}} R.
\end{equation*}
\itshape
\end{prop}

\begin{proof}
\noindent Using the second equation of the vorticity equation of the Castro-Lannes system \eqref{castro_lannes_formulation}, we have

\begin{equation*}
\partial_{t} \bm{\omega}_{y} + \epsilon u \partial_{x} \bm{\omega}_{y} + \frac{\epsilon}{\mu} \text{w} \partial_{z}  \bm{\omega}_{y} = \epsilon \bm{\omega}_{x} \partial_{x} v + \frac{\epsilon}{\sqrt{\mu}} \bm{\omega}_{z} \partial_{z} v + \frac{\epsilon}{\text{Ro} \sqrt{\mu}} \partial_{z} v.
\end{equation*}

\noindent Since $\bm{\omega}_{x} = -\frac{1}{\sqrt{\mu}} \partial_{z} v$ and $\bm{\omega}_{z} = \partial_{x} v$ we notice that $\epsilon \bm{\omega}_{x} \partial_{x} v + \frac{\epsilon}{\sqrt{\mu}} \bm{\omega}_{z} \partial_{z} v = 0$. Using Proposition \ref{equation_u} we get

\begin{equation*}
\partial_{t} \bm{\omega}_{y} + \epsilon \overline{u} \partial_{x} \bm{\omega}_{y} - \epsilon \partial_{x} \left[ \left(1 \hspace{-0.05cm} + \hspace{-0.05cm} z \hspace{-0.05cm} - \hspace{-0.05cm} \beta b \right) \hspace{-0.05cm} \overline{u} \right] \partial_{z}  \bm{\omega}_{y} - \frac{\epsilon}{\text{Ro} \sqrt{\mu}} \partial_{z} v =  \mu^{\frac{3}{2}} R,
\end{equation*}

\noindent Then, integrating with respect to $z$, using the fact that $\partial_{t} \zeta + \partial_{x} \left( h \overline{u} \right) = 0$ and $u_{\text{sh}} = - \int_{z}^{\epsilon \zeta} \bm{\omega}_{y}$, we get

\begin{equation*}
\partial_{t} u_{\text{sh}} + \epsilon \overline{u} \partial_{x} u_{\text{sh}} + \epsilon u_{\text{sh}} \partial_{x} \overline{u} + \frac{\epsilon}{\text{Ro} \sqrt{\mu}} \left(\underline{v} - v \right) = \epsilon \partial_{x} \left[ \left(1+z - \beta b \right) \overline{u} \right] \partial_{z}  u_{\text{sh}} + \mu^{\frac{3}{2}} R.
\end{equation*}

\noindent Integrating again with respect to $z$, using the fact that $\partial_{t} \zeta + \partial_{x} \left( h \overline{u} \right) = 0$ and $Q_{x} = \overline{u_{\text{sh}}^{\ast}}$, we obtain

\begin{equation*}
\partial_{t} Q_{x} + \epsilon \overline{u} \partial_{x} Q_{x} + \epsilon Q_{x} \partial_{x} \overline{u} + \frac{\epsilon}{\text{Ro} \sqrt{\mu}} \left(\underline{v} - \overline{v} \right) = \hspace{-0.05cm} \hspace{-0.05cm} \mu^{\frac{3}{2}} R.
\end{equation*}
\end{proof}

\noindent We have a similar equation for $\text{Q}_{y}$.

\begin{prop}\label{eq_Qy}
If \upshape $\left(\zeta, \Umuszero, \bm{\omega} \right)$ \itshape satisfy the Castro-Lannes system \eqref{castro_lannes_formulation}, then, in the Boussinesq regime \eqref{bouss_regime}, $\text{Q}_{x}$ satisfies the following equation

\upshape
\begin{equation*}
\partial_{t} \text{Q}_{y}  + \epsilon \overline{u} \partial_{x} \text{Q}_{y} \hspace{-0.1cm} + \epsilon \text{Q}_{x} \partial_{x} \overline{v} + \frac{\epsilon}{\text{Ro} \sqrt{\mu}} \left( \overline{u}  -  \underline{u} \right) = \mu^{\frac{3}{2}} R
\end{equation*}
\itshape

\noindent and  $\vastsh$ satisfies the equation

\upshape
\begin{equation*}
\partial_{t} \vastsh + \epsilon \overline{u} \partial_{x} \vastsh + \epsilon \uastsh \partial_{x} \overline{v} + \frac{\epsilon}{\text{Ro} \sqrt{\mu}} \left(u - \overline{u} \right) = \mu^{\frac{3}{2}} R.
\end{equation*}
\itshape

\end{prop}

\begin{proof}
\noindent Using the first equation of the vorticity equation of the Castro-Lannes system \eqref{castro_lannes_formulation}, we have

\begin{equation*}
\partial_{t} \bm{\omega}_{x} + \epsilon u \partial_{x} \bm{\omega}_{x} + \frac{\epsilon}{\mu} \text{w} \partial_{z}  \bm{\omega}_{x} = \epsilon \bm{\omega}_{x} \partial_{x} u + \frac{\epsilon}{\sqrt{\mu}} \bm{\omega}_{z} \partial_{z} u + \frac{\epsilon}{\text{Ro} \sqrt{\mu}} \partial_{z} u.
\end{equation*}

\noindent Then, using the fact that $\nabla^{\mu,0} \cdot \bm{\omega} = 0$ and $\nabla^{\mu,0} \cdot \textbf{U}^{\mu,\gamma}  = 0$, we get

\begin{equation*}
\partial_{t} \bm{\omega}_{x} - \frac{\epsilon}{\sqrt{\mu}} \partial_{z} \left( u \bm{\omega}_{z} \right) + \frac{\epsilon}{\mu} \partial_{z} \left( \text{w} \bm{\omega}_{x} \right) = \frac{\epsilon}{\text{Ro} \sqrt{\mu}} \partial_{z} u.
\end{equation*}

\noindent then, we integrate with respect to $z$ and, using the fact that $\partial_{t} \zeta - \frac{1}{\mu} \underline{\textbf{U}}^{\mu} \cdot \textbf{N}^{\mu,0} = 0$, $\bm{\omega}_{x} = - \frac{1}{\sqrt{\mu}} \partial_{z} v$ and $\bm{\omega}_{z} = \partial_{x} v$, we obtain

\begin{equation*}
\partial_{t} \left( \int_{-1+\beta b}^{\epsilon \zeta} \bm{\omega}_{x} \right) - \frac{\epsilon}{\sqrt{\mu}} \underline{u} \partial_{x} \underline{v} + \frac{\epsilon}{\sqrt{\mu}} u \partial_{x} v + \frac{\epsilon}{\mu^{\frac{3}{2}}} \text{w} \partial_{z} v + \frac{\epsilon}{\text{Ro} \sqrt{\mu}} \left(u - \underline{u} \right) = 0.
\end{equation*}

\noindent Then, we integrate again with respect to $z$ and, using Proposition \ref{equation_v} and the fact that $\partial_{t} \zeta - \frac{1}{\mu} \underline{\textbf{U}}^{\mu} \cdot \textbf{N}^{\mu,0} = 0$, $\textbf{U}_{b}^{\mu} \cdot \textbf{N}_{b}^{\mu,0} = 0$, and $\nabla^{\mu,0} \cdot \textbf{U}^{\mu} = 0$, we get

\begin{equation*}
\partial_{t} \text{Q}_{y} - \frac{\epsilon}{\sqrt{\mu}} \underline{u} \partial_{x} \underline{v} + \frac{\epsilon}{\sqrt{\mu}} \frac{1}{h} \partial_{x} \left( \int_{-1+\beta b}^{\epsilon \zeta} u v \right) + \frac{1}{\sqrt{\mu} h} \partial_{t} h \overline{v} + \frac{\epsilon}{\text{Ro} \sqrt{\mu}} \left(\overline{u} - \underline{u} \right) = 0.
\end{equation*}

\noindent Then, thanks to Propositions \ref{mean_eq}, \ref{equation_v} and \ref{equation_u}  we finally obtain that

\begin{equation*}
\partial_{t} \text{Q}_{y} \hspace{-0.1cm} + \hspace{-0.1cm} \epsilon \overline{u} \partial_{x} \text{Q}_{y} \hspace{-0.1cm} + \hspace{-0.1cm} \epsilon \text{Q}_{x} \partial_{x} \overline{v} \hspace{-0.1cm} + \hspace{-0.1cm} \frac{\epsilon}{\text{Ro} \sqrt{\mu}} \hspace{-0.1cm} \left( \overline{u} \hspace{-0.1cm} - \hspace{-0.1cm} \underline{u} \right) \hspace{-0.1cm} = \mu^{\frac{3}{2}} R.
\end{equation*}
\end{proof}

\noindent  Notice that we give in Subsection \ref{improve_Q} a generalization of the two previous propositions to the fully nonlinear Green-Naghdi regime. Furthermore, in the following proposition we give an equation for $\textbf{V}^{\sharp}$ up to terms of order $\mathcal{O} \left(\sqrt{\mu} \right)$.

\begin{prop}\label{udiese_eq}
If \upshape $\left(\zeta, \Umuszero, \bm{\omega} \right)$ \itshape satisfy the Castro-Lannes system \eqref{castro_lannes_formulation}, then \upshape $\textbf{V}^{\sharp}$ \itshape satisfies the following equation

\upshape
\begin{equation*}
\partial_{t} \textbf{V}^{\sharp} + \epsilon \textbf{V}^{\sharp} \partial_{x} \overline{u} + \epsilon \overline{u} \partial_{x} \textbf{V}^{\sharp} + \frac{\epsilon}{\text{Ro}} \textbf{V}^{\sharp \perp} = \max \left( \epsilon, \frac{\epsilon}{\text{Ro}} \right) \sqrt{\mu} R.
\end{equation*}
\itshape

\end{prop} 

\begin{proof}
\noindent The proof is similar to the computation in Part 4.4 in \cite{Castro_Lannes_shallow_water}. After multiplying by $(1+z-\beta b)^{2}$ and integrating with respect to $z$ the second equations of Propositions \ref{eq_Qx} and \ref{eq_Qy}, we neglect all the term of order $\mathcal{O}(\sqrt{\mu})$. Then, using the fact that $\partial_{t} \zeta + \partial_{x}(h \overline{u}) = 0$ and $\textbf{V} - \overline{\textbf{V}} = \sqrt{\mu} \textbf{V}_{\text{sh}}^{\ast} + \mu R$, we get the result.
\end{proof}

\subsection{The Boussinesq-Coriolis equations}

\noindent We can now establish the Boussinesq-Coriolis equations when $d=1$. The Boussinesq-Coriolis equations are the following system

\begin{equation}\label{boussinesq_eq}
\left\{
\begin{aligned}
&\partial_{t} \zeta + \partial_{x} \left( h \overline{u}  \right) = 0,\\
&\left(1- \frac{\mu}{3} \partial_{x}^{2} \right) \partial_{t} \overline{u} + \partial_{x} \zeta + \epsilon \overline{u} \partial_{x} \overline{u} - \frac{\epsilon}{\text{Ro}} \overline{v} + \frac{\epsilon}{\text{Ro}} \mu^{\frac{3}{2}} \frac{1}{24} \partial_{x}^{2} \frac{v^{\sharp}}{h} = 0,\\
&\partial_{t} \overline{v} + \epsilon \overline{u} \partial_{x} \overline{v} + \frac{\epsilon}{\text{Ro}} \overline{u} = 0,\\
&\partial_{t} \textbf{V}^{\sharp} + \epsilon \textbf{V}^{\sharp} \partial_{x} \overline{u} + \epsilon  \overline{u} \partial_{x} \textbf{V}^{\sharp} + \frac{\epsilon}{\text{Ro}} \textbf{V}^{\sharp \perp} = 0,
\end{aligned}
\right.
\end{equation}

\noindent where $\textbf{V}^{\sharp}$ is defined in \eqref{u_diese}. We can show that the Boussinesq-Coriolis equations are an order $\mathcal{O}(\mu^{2})$ approximation of the water waves equations.

\begin{remark}\label{renorm_v_sharp}
\noindent Inspired by \cite{lannes_marche}, we can renormalize \upshape$\textbf{V}^{\sharp}$ \itshape by $h$ and, using the first equation of \eqref{boussinesq_eq}, we get the following equation

\upshape
\begin{equation*}
\partial_{t} \left(\frac{\textbf{V}^{\sharp}}{h} \right) + \epsilon \overline{u} \partial_{x} \left(\frac{\textbf{V}^{\sharp}}{h} \right) + \frac{\epsilon}{\text{Ro}} \left(\frac{\textbf{V}^{\sharp}}{h} \right)^{\perp} = 0.
\end{equation*}
\itshape

\noindent This remark will be useful for the local existence (Proposition \ref{existence_boussi}).
\end{remark}

\begin{prop}\label{constit_boussi}
In the  Boussinesq regime $\mathcal{A}_{Bouss}$ \eqref{bouss_regime}, the Castro-Lannes equations \eqref{castro_lannes_formulation} are consistent at order $\mathcal{O}(\mu^{2})$ with the Boussinesq-Coriolis equations \eqref{boussinesq_eq} in the sense of Definition \ref{consistence}.
\end{prop}

\begin{proof}
\noindent The first equation of the Boussinesq-Coriolis equations is always satisfied for a solution of the Castro-Lannes formulation by Proposition \ref{mean_eq}. For the second equation, we use Proposition \ref{equation_u}, Proposition \ref{eq_Qx} together with Proposition \ref{equation_partial_t_u}, Lemma \ref{udiese_T_eq} and Proposition \ref{udiese_eq} (we recall that $\epsilon = \mathcal{O} (\mu)$). Notice the fact that all the terms with $\text{Q}_{x}$ disappear. We also use the fact that

\begin{equation*}
h^{3} v^{\sharp} = \frac{v^{\sharp}}{h} + \mu R.
\end{equation*}

\noindent Then, the third equation follows from Proposition \ref{equation_u}, \ref{equation_u} and \ref{eq_Qy} (all the terms with $\text{Q}_{y}$ also disappear).
\end{proof}

\noindent We notice that contrary to the classical Boussinesq equations, we have a new term due to the vorticity that we can not neglect in presence of a Coriolis forcing. In our knowledge, this term was not highlighted before in the literature.

\begin{remark}\label{term_vort_strange}
\noindent In the Boussinesq-Coriolis system \eqref{boussinesq_eq} we could simplify the term $\partial_{x}^{2} \frac{v^{\sharp}}{h}$ by $\partial_{x}^{2} v^{\sharp}$ since these terms are equal up to a remainder of order $\mathcal{O}(\mu)$. However, the term $\partial_{x}^{2} \frac{v^{\sharp}}{h}$ will be essential for the local existence \upshape(\itshape see Remark \ref{term_vort_strange_explanation} \upshape)\itshape.
\end{remark}

\begin{remark}\label{weak_rot_boussi}
\noindent If we assume that $\frac{\epsilon}{\text{Ro}} = \mathcal{O} \left( \sqrt{\mu} \right)$, we can neglect the term with $v^{\sharp}$ in the second equation of \eqref{boussinesq_eq} and we obtain

\upshape
\begin{equation}\label{weak_rot_boussi_eq}
\left\{
\begin{aligned}
&\partial_{t} \zeta + \partial_{x} \left( h \overline{u}  \right) = 0,\\
&\left(1- \frac{\mu}{3} \partial_{x}^{2} \right) \partial_{t} \overline{u} + \partial_{x} \zeta + \epsilon \overline{u} \partial_{x} \overline{u} - \frac{\epsilon}{\text{Ro}} \overline{v} = 0,\\
&\partial_{t} \overline{v} + \epsilon \overline{u} \partial_{x} \overline{v} + \frac{\epsilon}{\text{Ro}} \overline{u} = 0.
\end{aligned}
\right.
\end{equation}
\itshape

\noindent This system is the classical Boussinesq equations with a standard Coriolis forcing. It is consistent of order $\mathcal{O}(\mu^{2})$ with the Boussinesq-Coriolis equations \eqref{boussinesq_eq}. We use this system in Subsections \ref{model_ostrov} and \ref{model_kdv}.
\end{remark}

\subsection{Full justification of the Boussinesq-Coriolis equations}

\noindent In this part, we fully justify the Boussinesq-Coriolis equations \eqref{boussinesq_eq}. In the following we denote by $u$ the quantity $\overline{u}$ and by $v$ the quantity $\overline{v}$. We show that the Boussinesq-Coriolis equations are wellposed. We define the energy space

\begin{equation}
X^{s}(\R) =  H^{s}(\R) \times H^{s+1}(\R) \times H^{s}(\R) \times H^{s+1}(\R) \times H^{s+1}(\R),
\end{equation}

\noindent endowed with the norm 

\begin{equation}
\lver (\zeta,u,v,\textbf{W}) \rver_{X^{s}_{\mu}}^{2} =  \lver \zeta \rver^{2}_{H^{s}} +  \lver u \rver^{2}_{H^{s}} +  \mu \lver \partial_{x} u \rver^{2}_{H^{s}} +  \lver v \rver^{2}_{H^{s}} +  \lver \textbf{W} \rver^{2}_{H^{s}} + \mu \lver \partial_{x} \textbf{W} \rver^{2}_{H^{s}}.
\end{equation}

\begin{prop}\label{existence_boussi}
\noindent Let $A > 0$, $s > \frac{1}{2} + 1$, \upshape $ \left( \zeta_{0},u_{0},v_{0},\textbf{V}^{\sharp}_{0} \right) \in X^{s}(\R)$ \itshape and \upshape $b \in H^{s+1}(\R)$\itshape. We suppose that \upshape$\left( \epsilon,\beta,\gamma,\mu,\text{Ro} \right) \in \mathcal{A}_{Bouss}$ \itshape.  We assume that

\upshape
\begin{equation*}
\exists \, h_{\min}  > 0 \text{ ,  } \epsilon \zeta _{0} + 1 - \beta b \geq h_{\min}
\end{equation*}
\itshape

\noindent and 

\upshape
\begin{equation*}
\lver \left( \zeta_{0},u_{0},v_{0}, \frac{\textbf{V}^{\sharp}_{0}}{1+\epsilon \zeta_{0} - \beta b} \right) \rver_{X^{s}_{\mu}} + \lver b \rver_{H^{s+1}} \leq A.
\end{equation*}
\itshape
 
\noindent Then, there exists an existence time \upshape$T > 0$ \itshape and a unique solution \upshape$\left( \zeta, u, v, \textbf{V}^{\sharp} \right)$ \itshape on $[0,T]$ to the Boussinesq-Coriolis equations \eqref{boussinesq_eq} with initial data \upshape$\left( \zeta_{0}, u_{0}, v_{0}, \textbf{V}^{\sharp}_{0} \right)$ \itshape such that we have \small \upshape$\left( \zeta, u, v, \frac{\textbf{V}^{\sharp}}{h} \right) \in \mathcal{C} \left([0,T]; X^{s}(\R) \right)$ \itshape \normalsize with $h=1+\epsilon \zeta - \beta b$. Moreover, 

\upshape
\begin{equation*}
T =  \frac{T_{0}}{\max(\mu, \frac{\epsilon}{\text{Ro}} \sqrt{\mu})} \text{    ,    }  \frac{1}{T_{0}} = c^{1} \text{   and   }\underset{[0,T]}{\max} \lver \left(\zeta,u,v, \frac{\textbf{V}^{\sharp}}{h} \right)(t, \cdot) \rver_{X^{s}_{\mu}} = c^{2},
\end{equation*}
\itshape

\noindent with \upshape $c^{j} = C \left(A, \mu_{\max}, \frac{1}{h_{\min}} \right)$.\itshape

\end{prop}

\begin{proof}
\noindent We only give the energy estimates. For the existence see for instance the proof of Theorem 1 in \cite{israwi_green_naghdi}. We assume that $\left( \zeta, u, v, \textbf{V}^{\sharp} \right) $ solves \eqref{boussinesq_eq} on $\left[0,\frac{T_{0}}{\max(\mu, \frac{\epsilon}{\text{Ro}} \sqrt{\mu})} \right]$ and that

\begin{equation*}
1+\epsilon \zeta - \beta b \geq \frac{h_{\min}}{2} \text{   on   } \left[0,\frac{T_{0}}{\max(\mu, \frac{\epsilon}{\text{Ro}} \sqrt{\mu})} \right].
\end{equation*}

\noindent We denote $U = \left(\zeta,u,v \right)^{t}$ and we focus first on the first three equations. This part is a small adaptation  of the proof of Theorem 1 in \cite{israwi_green_naghdi}. The the first three equations of the Boussinesq-Coriolis equations can be symmetrized, as an hyperbolic system, by multiplying the second and the third equations by $h=1+\epsilon \zeta - \beta b$. Then, we obtain the following system

\begin{equation*}
\mathcal{A}_{0}(U) \partial_{t} U + \mathcal{A}_{1}(U) \partial_{x} U + B_{1} U + \frac{\epsilon}{\text{Ro}} B_{2}(U) U = \frac{\epsilon}{\text{Ro}} \mu^{\frac{3}{2}} F(h,v^{\sharp}),
\end{equation*}

\noindent where

\begin{equation*}
\mathcal{A}_{0}(U) = \begin{pmatrix} 1 & 0 & 0 \\ 0 & h - \mu \frac{h}{3} \partial_{x}^{2} & 0 \\ 0 & 0 & h \end{pmatrix} \text{, } \mathcal{A}_{1}(U) = \begin{pmatrix} \epsilon u & h & h\\ h & \epsilon hu & 0 \\ h & 0 & \epsilon hu \end{pmatrix}
\end{equation*}

\noindent and

\begin{equation*}
B_{1} = \begin{pmatrix} 0 & -\beta \partial_{x} b & 0 \\ 0 & 0 & 0 \\ 0 & 0 & 0 \end{pmatrix} \text{   ,   } B_{2}(U) = \begin{pmatrix} 0 & 0 & 0 \\ 0 & 0 & -h \\ 0 & h & 0 \end{pmatrix} \text{   and   } F(h,v^{\sharp}) = \begin{pmatrix} 0 \\ - \frac{h}{24} \partial_{x}^{2} \frac{v^{\sharp}}{h} \\ 0 \end{pmatrix}.
\end{equation*}

\noindent Then we remark that $\mathcal{A}_{1}$ is symmetric and there exists $c_{1},c_{2} = C \left( \frac{1}{h_{\min}}, \lver h \rver_{L^{\infty}} \right)$ such that

\begin{equation*}
c_{1} \lver \partial_{x} f \rver^{2}_{2} \leq \left( - \frac{1}{3} \partial_{x} \left(h \partial_{x} f \right) ,f \right)_{2}  \leq c_{2} \lver \partial_{x} f \rver^{2}_{2}.
\end{equation*}

\noindent Hence we introduce the symmetric matrix operator

\begin{equation*}
S(U) = \begin{pmatrix} 1 & 0 & 0 \\ 0 & h - \frac{\mu}{3} \partial_{x} \left( h \partial_{x} \cdot \right) & 0 \\ 0 & 0 & h \end{pmatrix}
\end{equation*}
 
\noindent and the energy associated

\begin{equation*}
\mathcal{E}^{s}(U) = \left(S(U) \Lambda^{s} U, \Lambda^{s} U \right)_{2}.
\end{equation*}

\noindent Then, we see that

\begin{equation*}
\left(\Lambda^{s} B_{2}(U) U, \Lambda^{s} U \right)_{2} = 0
\end{equation*}

\noindent and  by standard product estimates we get

\begin{equation}\label{vsharp-estim_prob}
\mu^{\frac{3}{2}} \lver \left(h \Lambda^{s} \partial_{x}^{2} \frac{v^{\sharp}}{h}, \Lambda^{s} u \right)_{2} \rver \leq \sqrt{\mu} C(\mathcal{E}^{s}(U), \lver b \rver_{H^{s+1}}) \sqrt{\mu} \lver \frac{v^{\sharp}}{h} \rver_{H^{s+1}}. 
\end{equation}

\noindent Furthermore, notice that

\begin{align*}
\mu \lver \partial_{t} \partial_{x} u \rver_{H^{s}} &= \mu \lver \hspace{-0.05cm} \left(1- \frac{\mu}{3} \partial_{x}^{2} \right)^{-1} \partial_{x} \left( \partial_{x} \zeta + \epsilon u \partial_{x} u - \frac{\epsilon}{\text{Ro}} v + \frac{\epsilon}{\text{Ro}} \frac{\mu^{\frac{3}{2}}}{24} \partial_{x}^{2} \frac{v^{\sharp}}{h} \hspace{-0.05cm}\right) \hspace{-0.05cm} \rver_{H^{s}},\\
&\leq C \hspace{-0.05cm} \left(\mu_{\max}, \mathcal{E}^{s}(U), \sqrt{\mu} \lver \partial_{x} \frac{v^{\sharp}}{h} \rver_{H^{s}} \right).
\end{align*}

\noindent and therefore

\begin{equation*}
\left(\frac{\mu}{3} \partial_{x} h \Lambda^{s} \partial_{x} \partial_{t} u, \Lambda^{s} u \right)_{2} \leq  \mu C \left( \mathcal{E}^{s}(U), \lver b \rver_{H^{s+1}}, \sqrt{\mu} \lver \partial_{x} \frac{v^{\sharp}}{h} \rver_{H^{s}} \right).
\end{equation*}

\noindent Gathering all the previous estimate and proceeding as in \cite{israwi_green_naghdi} we obtain

\begin{equation*}
\frac{d}{dt} \mathcal{E}^{s}(U) \leq \max \left( \mu, \frac{\epsilon}{\text{Ro}} \sqrt{\mu} \right) C \left(\mathcal{E}^{s}(U), \lver b \rver_{H^{s+1}}, \lver \frac{v^{\sharp}}{h} \rver_{H^{s}}, \sqrt{\mu} \lver \partial_{x} \frac{v^{\sharp}}{h} \rver_{H^{s}} \right).
\end{equation*}

\noindent Furthermore, using Remark \ref{renorm_v_sharp} and the Kato-Ponce estimate, we get

\begin{equation*}
\begin{aligned}
&\frac{d}{dt} \lver  \frac{\textbf{V}^{\sharp}}{h} \rver_{H^{s}}^{2} \leq \mu C \lver u \rver_{H^{s}} \lver  \frac{\textbf{V}^{\sharp}}{h} \rver_{H^{s}}^{2},\\
&\frac{d}{dt} \mu \lver \partial_{x}  \frac{\textbf{V}^{\sharp}}{h} \rver_{H^{s}}^{2} \leq \mu C \left(\sqrt{\mu} \lver \partial_{x} u \rver_{H^{s}} \lver \frac{\textbf{V}^{\sharp}}{h} \rver_{H^{s}}^{2} + \lver u \rver_{H^{s}} \sqrt{\mu} \lver \partial_{x} \frac{\textbf{V}^{\sharp}}{h} \rver_{H^{s}} \right) \sqrt{\mu} \lver \partial_{x} \frac{\textbf{V}^{\sharp}}{h} \rver_{H^{s}}.
\end{aligned}
\end{equation*}

\noindent Then, the result follows.
\end{proof}

\begin{remark}\label{term_vort_strange_explanation}
\noindent Notice that the previous energy estimates do not imply that \upshape$\textbf{V}^{\sharp} \in H^{s+1}(\R)$. \itshape Hence, it is essential that in Inequality \eqref{vsharp-estim_prob} we have the term $\partial_{x}^{2} \frac{v^{\sharp}}{h}$ and not simply $\partial_{x}^{2} v^{\sharp}$ \upshape(\itshape see Remark \ref{term_vort_strange}\upshape)\itshape.
\end{remark}

\noindent Then, we similarly can prove a local wellposedness result for System \eqref{weak_rot_boussi_eq}. 

\begin{cor}\label{existence_weak_rot_boussi}
\noindent Let $A > 0$, $s > \frac{1}{2} + 1$, \upshape $ \left( \zeta_{0},u_{0},v_{0} \right) \in H^{s} (\R) \times H^{s+1} (\R) \times H^{s} (\R)$ \itshape and \upshape $b \in H^{s+1}(\R)$\itshape. We suppose that \upshape$\left( \epsilon,\beta,\gamma,\mu,\text{Ro} \right) \in \mathcal{A}_{Bouss}$ \itshape.  We assume that

\upshape
\begin{equation*}
\exists \, h_{\min}  > 0 \text{ ,  } \epsilon \zeta _{0} + 1 - \beta b \geq h_{\min}
\end{equation*}
\itshape

\noindent and 

\upshape
\begin{equation*}
\lver \zeta_{0} \rver_{H^{s}} + \lver u_{0} \rver_{H^{s}} + \sqrt{\mu} \lver \partial_{x} u_{0} \rver_{H^{s}} + \lver v_{0} \rver_{H^{s}} + \lver b \rver_{H^{s+1}} \leq A.
\end{equation*}
\itshape
 
\noindent Then, there exists an existence time \upshape$T > 0$ \itshape and a unique solution to the Boussinesq-Coriolis equations \eqref{boussinesq_eq} \small \upshape$\left( \zeta, u, v \right) \in \mathcal{C} \left([0,T]; H^{s} (\R) \times H^{s+1} (\R) \times H^{s} (\R) \right)$ \itshape \normalsize with initial data \upshape$\left( \zeta_{0}, u_{0}, v_{0} \right)$\itshape. Moreover, 

\upshape
\begin{equation*}
T =  \frac{T_{0}}{\mu} \text{    ,    }  \frac{1}{T_{0}} = c^{1} \text{   and   } \hspace{0.2cm} \underset{[0,T]}{\max} \hspace{0.2cm} \lver \zeta(t,\cdot) \rver_{H^{s}} + \lver u(t,\cdot) \rver_{H^{s}} + \sqrt{\mu} \lver \partial_{x} u(t,\cdot) \rver_{H^{s}} + \lver v(t,\cdot) \rver_{H^{s}} = c^{2},
\end{equation*}
\itshape

\noindent with \upshape $c^{j} = C \left(A, \mu_{\max}, \frac{1}{h_{\min}} \right)$.\itshape

\end{cor}

\noindent Furthermore, we have a stability result for the Boussinesq-Coriolis system \eqref{boussinesq_eq}.

\begin{prop}\label{stability_boussi}
\noindent Let the assumptions of Proposition \ref{existence_boussi} satisfied. Suppose that there exists  \upshape$\left( \tilde{\zeta}, \tilde{u}, \tilde{v}, \frac{\tilde{\textbf{V}}^{\sharp}}{\tilde{h}} \right) \in \mathcal{C} \left(\left[0,\frac{T_{0}}{\max \left(\mu, \frac{\epsilon \sqrt{\mu}}{\text{Ro}} \right)} \right]; X^{s}(\R) \right)$ \itshape satisfying

\upshape
\begin{equation*}
\left\{
\begin{aligned}
&\partial_{t} \tilde{\zeta} + \partial_{x} \left( \tilde{h} \tilde{u}  \right) = R_{1},\\
&\left(1- \frac{\mu}{3} \partial_{x}^{2} \right) \partial_{t} \tilde{u} + \partial_{x} \tilde{\zeta} + \epsilon \tilde{u} \partial_{x} \tilde{u} - \frac{\epsilon}{\text{Ro}} \tilde{v} + \frac{\epsilon}{\text{Ro}} \mu^{\frac{3}{2}} \frac{1}{24} \partial_{x} \frac{\tilde{v}^{\sharp}}{\tilde{h}} = R_{2},\\
&\partial_{t} \tilde{v} + \epsilon \tilde{u} \partial_{x} \tilde{v} + \frac{\epsilon}{\text{Ro}} \tilde{u} = R_{3},\\
&\partial_{t} \frac{\tilde{\textbf{V}}^{\sharp}}{\tilde{h}} + \epsilon + \epsilon \tilde{u} \partial_{x} \frac{\tilde{\textbf{V}}^{\sharp}}{\tilde{h}} + \frac{\epsilon}{\text{Ro}} \frac{\tilde{\textbf{V}}^{\sharp}}{\tilde{h}} = R_{4},
\end{aligned}
\right.
\end{equation*}
\itshape

\noindent where $\tilde{h} = 1 + \epsilon \tilde{\zeta} - \beta b$ and with $R = (R_{1}, R_{2}, R_{3}, R_{4}) \in L^{\infty} \left( \left[0, \frac{T_{0}}{\max \left(\mu, \frac{\epsilon \sqrt{\mu}}{\text{Ro}} \right)} \right] ;X^{s}(\R) \right)$. Then, if we denote \upshape $\mathfrak{e} = \left( \zeta, u, v, \textbf{V}^{\sharp} \right) - \left( \tilde{\zeta}, \tilde{u}, \tilde{v}, \tilde{\textbf{V}}^{\sharp} \right)$ \itshape where \upshape $\left( \zeta, u, v, \textbf{V}^{\sharp} \right)$ \itshape is the solution given in Proposition \ref{existence_boussi}, we have

\upshape
\small
\begin{equation*}
\lver \mathfrak{e}(t) \rver_{X^{s-1}_{\mu}} \leq C \left(A, \mu_{\max}, \frac{1}{h_{\min}}, \lver \left( \tilde{\zeta}, \tilde{u}, \tilde{v}, \frac{\tilde{\textbf{V}}^{\sharp}}{\tilde{h}}, R \right) \rver_{L^{\infty} \left([0,t]; X_{\mu}^{s} \times X^{s}_{\mu} \right)} \right) \hspace{-0.1cm} \left(\lver \mathfrak{e}_{|t=0} \rver_{X^{s-1}_{\mu}} + t \lver R \rver_{X_{\mu}^{s}} \right).
\end{equation*}
\normalsize
\itshape

\end{prop}

\begin{proof}
\noindent This proof is a small adaptation of the one of Proposition 6.5 in \cite{Lannes_ww} (see also \cite{Alvarez_Lannes}). We denote $\tilde{U} = \left( \tilde{\zeta}, \tilde{u}, \tilde{v} \right)$, $\mathfrak{e}_{a} = U - \tilde{U}$, $R_{a} = (R_{1}, R_{2}, R_{3})$ and we keep the notations of the proof of Proposition \ref{existence_boussi}. Since the Boussinesq-Coriolis equations are symmetrizable, we have

\begin{equation*}
\left\{
\begin{aligned}
&\mathcal{A}_{0}(U) \partial_{t} \mathfrak{e}_{a} + \mathcal{A}_{1}(U) \partial_{x} \mathfrak{e}_{a}  + B_{1} \mathfrak{e}_{a} + \frac{\epsilon}{\text{Ro}} B_{2}(U)  \mathfrak{e}_{a} = \frac{\epsilon}{\text{Ro}} \mu^{\frac{3}{2}} F(h, v^{\sharp} - \tilde{v}^{\sharp}) + G,\\
&\partial_{t} \left(\frac{\textbf{V}^{\sharp}}{h} - \frac{\tilde{\textbf{V}}^{\sharp}}{\tilde{h}} \right) + \epsilon u \partial_{x} \left(\frac{\textbf{V}^{\sharp}}{h} - \frac{\tilde{\textbf{V}}^{\sharp}}{\tilde{h}} \right) + \frac{\epsilon}{\text{Ro}} \left(\frac{\textbf{V}^{\sharp}}{h} - \frac{\tilde{\textbf{V}}^{\sharp}}{\tilde{h}}\right)^{\perp} = H,
\end{aligned}
\right.
\end{equation*}

\noindent where

\begin{equation*}
\begin{aligned}
G = &F(h,\tilde{v}^{\sharp}) - F(\tilde{h},\tilde{v}^{\sharp}) - R_{a} - (\mathcal{A}_{0}(U) - \mathcal{A}_{0}(\tilde{U}) ) \partial_{t} \tilde{U} \\
&- (\mathcal{A}_{1}(U) - \mathcal{A}_{1}(\tilde{U})) \partial_{x} \tilde{U} - \frac{\epsilon}{\text{Ro}}(B_{2}(U) - B_{2}(\tilde{U}))U,\\
& \hspace{-0.7cm} H= \epsilon (\tilde{u} - u) \partial_{x} \frac{\tilde{\textbf{V}}^{\sharp}}{\tilde{h}} + R_{4}.\\
\end{aligned}
\end{equation*}

\noindent Then, using standard products estimates, we get (notice that $s > \frac{1}{2} + 1$)

\small
\begin{equation*}
\left( \Lambda^{s-1} G, \Lambda^{s-1} \mathfrak{e}_{a} \right)_{2} \leq \left( \lver R \rver_{X^{s}_{\mu}} + \mu C \left(\mathcal{E}^{s}(\tilde{U}), \mathcal{E}^{s-1}(\partial_{t} \tilde{U}), \lver \frac{\tilde{v}^{\sharp}}{\tilde{h}} \rver_{H^{s}}, \sqrt{\mu} \lver \partial_{x} \frac{\tilde{v}^{\sharp}}{\tilde{h}} \rver_{H^{s}} \right) \lver \mathfrak{e} \rver_{X^{s-1}} \right) \lver \mathfrak{e} \rver_{X^{s-1}}
\end{equation*}
\normalsize

\noindent and

\small
\begin{equation*}
\left(\Lambda^{s-1} H,\Lambda^{s-1} \left( \frac{\textbf{V}^{\sharp}}{h} - \frac{\tilde{\textbf{V}}^{\sharp}}{\tilde{h}} \right) \right)_{2} \hspace{-0.2cm} \leq \hspace{-0.1cm} \left(\hspace{-0.1cm} \lver R \rver_{X^{s}_{\mu}} + \mu C \left(\mathcal{E}^{s} (\tilde{U}), \lver \frac{\tilde{\textbf{V}}^{\sharp}}{\tilde{h}} \rver_{H^{s}}, \sqrt{\mu} \lver \partial_{x} \frac{\tilde{v}^{\sharp}}{\tilde{h}} \rver_{H^{s}} \right) \lver \mathfrak{e} \rver_{X^{s-1}} \right) \lver \mathfrak{e} \rver_{X^{s-1}} \hspace{-0.05cm}.
\end{equation*}
\normalsize

\noindent Then, the result follows from energy estimates and the Gronwall's lemma.
\end{proof}

\noindent The two previous results and  Theorem \ref{existence_ww} allow us to fully justify the Boussinesq-Coriolis equations. We recall that the operators $\overline{\textbf{V}} [\epsilon \zeta_{0}, \beta b](\Umuszero,\bm{\omega})$ and $\textbf{V}_{\text{sh}}[\epsilon \zeta, \beta b](\Umuszero,\bm{\omega})(t,X)$ are defined in \eqref{V_average} and \eqref{V_shear} respectively.

\begin{thm}\label{boussi_water_wave_compare}
\noindent Let \upshape$\textbf{N} \geq 7$ and \upshape$\left( \epsilon,\beta,\gamma,\mu,\text{Ro} \right) \in \mathcal{A}_{Bouss}$ \itshape. We assume that we are under the assumptions of Theorem \ref{existence_ww}. Then, we can define the following quantity \upshape $\left(u_{0},v_{0} \right)^{t} = \overline{\textbf{V}} [\epsilon \zeta_{0}, \beta b]((\Umuszero)_{0},\bm{\omega}_{0})$, $\left(u,v \right)^{t}  = \overline{\textbf{V}} [\epsilon \zeta, \beta b](\Umuszero,\bm{\omega})$, $\textbf{V}^{\sharp}_{0} = \textbf{V}^{\sharp} [\epsilon \zeta_{0}, \beta b]((\Umuszero)_{0},\bm{\omega}_{0})$, $\textbf{V}^{\sharp} = \textbf{V}^{\sharp} [\epsilon \zeta, \beta b](\Umuszero,\bm{\omega}_{0})$, \itshape and there exists a time $T > 0$ such that

\medskip

\noindent (i) $T$ has the form

\begin{equation*}
T = \frac{T_{0}}{\max(\mu, \frac{\epsilon}{\text{Ro}})},\text{ and }  \frac{1}{T_{0}} =c^{1}.
\end{equation*}
\itshape

\medskip

\noindent (ii) There exists a unique classical solution \small \upshape$\left( \zeta_{B}, u_{B}, v_{B}, \textbf{V}^{\sharp}_{B} \right)$ \itshape \normalsize of \eqref{boussinesq_eq} with the initial data \upshape$\left(\zeta_{0}, u_{0}, v_{0}, \textbf{V}^{\sharp}_{0} \right)$ \itshape  on $\left[ 0,T \right]$.

\medskip

\noindent (iii)  There exists a unique classical solution \upshape$\left(\zeta,\Umuszero, \bm{\omega} \right)$ \itshape of System \eqref{castro_lannes_formulation} with initial data \upshape$\left(\zeta_{0}, (\Umuszero)_{0}, \bm{\omega}_{0} \right)$ \itshape on $\left[ 0,T \right]$.

\medskip

\noindent (iv) The following error estimate holds, for $0 \leq t \leq T$,

\upshape
\begin{equation*}
\lver \left(\zeta,u,v, \textbf{V}^{\sharp} \right) - \left(\zeta_{B},u_{B}, v_{B}, \textbf{V}^{\sharp}_{B} \right) \rver_{L^{\infty}([0,t] \times \R)} \leq \mu^{2} t\, c^{2},
\end{equation*}
\itshape

\noindent with \upshape $c^{j} = C \left(A, \mu_{\max}, \frac{1}{h_{\min}}, \frac{1}{\mathfrak{a}_{\min}},\left\lvert b \right\rvert_{H^{N+2}} \right)$.\itshape
\end{thm}

\noindent This theorem shows that the solutions of the water waves system \eqref{castro_lannes_formulation} remain close to the solutions of the Boussinesq-Coriolis equations \eqref{boussinesq_eq} over times $\mathcal{O} \left(\frac{1}{\max(\mu, \frac{\epsilon}{\text{Ro}})} \right)$ with an accuracy of order $\mathcal{O}(\mu)$. Hence, if one considers a system and wants to show that the solutions of this system remain close to the solutions of the waves equations over times $\mathcal{O} \left(\frac{1}{\max(\mu, \frac{\epsilon}{\text{Ro}})} \right)$ with an accuracy of order $\mathcal{O}(\mu)$, it is sufficient to compare the solutions of this system with the solutions of the Boussinesq-Coriolis equations \eqref{boussinesq_eq}. It is our approach in the following.

\section{Different asymptotic models in the Boussinesq regime over a flat bottom}\label{asympto_model}

\noindent The Boussinesq-Coriolis equations \eqref{boussinesq_eq} are particularly interesting for the evolution of offshore water waves. Without vorticity, we get the so-called Boussinesq equations. When we add a rotation, and in particular Coriolis effects, a standard assumption made by physicists is to also assume that the Rossby radius, or Obukhov radius, $\frac{\sqrt{gH}}{f}$ is greater than the typical length of the waves $L$ (see for instance \cite{pedlosky}, \cite{gill}, \cite{leblond}). Then, different regimes for the Coriolis parameter were considered depending on whether the rotation is weak or not (\cite{ostrovsky_original}, \cite{germain-renouard}, \cite{grimshaw_ostrov_surf}). In this paper, we consider three different regimes (noticed in \cite{germain-renouard}), a strong rotation ($\frac{\epsilon}{\text{Ro}} \leq 1$), weak rotation ($\frac{\epsilon}{\text{Ro}} = \mathcal{O} (\sqrt{\mu})$) and very weak rotation ($\frac{\epsilon}{\text{Ro}} = \mathcal{O} (\mu)$). We derive and fully justify  different asymptotic models when the bottom is flat : a linear equation admitting the so-called Poincar\'e waves \eqref{poincare_wave} ; the Ostrovsky equation \eqref{ostrov_eq}, which is a generalization of the KdV equation \eqref{kdv_eq} in presence of a Coriolis forcing, when the rotation is weak; and the KdV equation when the rotation is very weak.

\subsection{Strong rotation, the Poincar\'e waves}\label{model_poincare}

\noindent In this part we are interested in the behaviour of long water waves under a strong Coriolis forcing (in the sense of \cite{germain-renouard}). We suppose that $\frac{\epsilon}{\text{Ro}}$ is of order $1$. The asymptotic regime is 

\begin{equation}\label{poincare_regime}
\mathcal{A}_{\text{Poin}} = \left\{ \left( \epsilon, \beta, \gamma, \mu, \text{Ro} \right), 0 \leq \mu \leq \mu_{0}, \epsilon = \mu, \beta = \gamma = 0, \frac{\epsilon}{\text{Ro}} = 1 \right\}.
\end{equation}

\noindent Then, the Boussinesq-Coriolis equations \eqref{boussinesq_eq} become

\begin{equation}\label{boussinesq_eq_poin}
\left\{
\begin{aligned}
&\partial_{t} \zeta + \partial_{x} \left( \left(1+ \mu \zeta \right) u \right) = 0,\\
&\left(1- \frac{\mu}{3} \partial_{x}^{2} \right) \partial_{t} u + \partial_{x} \zeta + \mu u \partial_{x} u - v + \frac{ \mu^{\frac{3}{2}} }{24} \partial_{x}^{2} \frac{v^{\sharp}}{h} = 0,\\
&\partial_{t} v + \mu u \partial_{x} v +  u = 0,\\
&\partial_{t} \textbf{V}^{\sharp} + \mu \textbf{V}^{\sharp} \partial_{x} u + \mu u \partial_{x}  \textbf{V}^{\sharp} + \textbf{V}^{\sharp \perp} = 0.
\end{aligned}
\right.
\end{equation}

\noindent Our purpose is to justify the so-called Poincar\'e waves or Sverdrup waves (\cite{sverdrup}), which are inertia-gravity waves in the linear setting. Dropping all the terms of order $\mathcal{O} \left( \mu \right)$ in the Boussinesq-Coriolis equation, we get the linear system

\begin{equation}\label{boussines_linear}
\left\{
\begin{aligned}
&\partial_{t} \zeta + \partial_{x} u = 0,\\
&\partial_{t} u + \partial_{x} \zeta - v = 0,\\
&\partial_{t} v +  u = 0.\\
\end{aligned}
\right.
\end{equation}

\noindent Then, if we denote $U = \left(\zeta, u, v \right)^{t}$, by taking the Fourier transform, we get

\begin{equation*}
\partial_{t} \widehat{U} = \mathcal{A} \widehat{U}  \text{   with   }   \mathcal{A} = \begin{pmatrix}
0 & - i \xi & 0 \\ - i \xi & 0 & 1 \\ 0 & -1 & 0 \end{pmatrix}
\end{equation*}

\noindent and we obtain,

\begin{equation}\label{poincare_wave}
\widehat{U} = \mathcal{S}(t,\xi) \widehat{U}_{0} = \begin{pmatrix} \frac{\xi^2 \cos(\sqrt{\xi^2+1} t)+1}{\xi^2+1} & -i\xi \frac{\sin(\sqrt{\xi^2+1}t)}{\sqrt{\xi^2+1}} & i \xi \frac{\cos(\sqrt{\xi^2+1} t) -1}{\xi^2+1} \\ -i \xi \frac{\sin(\sqrt{\xi^2+1}t)}{\sqrt{\xi^2+1}} & \cos(\sqrt{\xi^2+1} t) & \frac{\sin(\sqrt{\xi^2+1} t)}{\sqrt{\xi^2+1}} \\ -i \xi \frac{\cos(\sqrt{\xi^2+1} t)-1}{\xi^2+1} & - \frac{\sin(\sqrt{\xi^2+1} t)}{\sqrt{\xi^2+1}} & \frac{\xi^2+\cos(\sqrt{\xi^2+1} t)}{\xi^2+1} \end{pmatrix} \widehat{U}_{0}. 
\end{equation}

\noindent Commonly, Poincar\'e waves are waves of the form

\begin{equation*}
U(t,x) = e^{i(x k \pm t \sqrt{k^{2} + 1})} U_{0}.
\end{equation*}

\noindent They are solutions of the Klein-Gordon equation. In this setting, Poincar\'e waves correspond to solutions of System \eqref{boussines_linear} of the form

\begin{equation*}
\widehat{U}(t,\xi) = e^{i \pm t \sqrt{\xi^{2} + 1}} \widehat{U}_{0}(\xi).
\end{equation*}

\noindent Therefore, a solution of System \eqref{boussines_linear} is a sum of two Poincar\'e waves if and only if

\begin{equation*}
\begin{pmatrix} \frac{1}{\xi^2+1} & 0 &  -\frac{i \xi}{\xi^2+1} \\ 0 & 0 & 0 \\ \frac{i \xi}{\xi^2+1} & 0 & \frac{\xi^2}{\xi^2+1} \end{pmatrix} \widehat{U}_{0} = 0, 
\end{equation*}

\noindent which is equivalent to

\begin{equation}\label{poincare_ci}
\zeta_{0} = \partial_{x} v_{0}.
\end{equation}

\noindent In the following, we denote by $\mathcal{S}(t)$ the semi-group of the linear Boussinesq-Coriolis equation. The end of this part is devoted to the full justification of Poincar\'e waves. The following lemma shows that Condition \eqref{poincare_ci} is propagated by the flow of System \eqref{boussines_linear}.

\begin{lemma}\label{propa_condition_poincare}
\noindent Let $(\zeta,u,v)$ be a solution of \eqref{boussines_linear} such that $(\zeta,u,v)_{|t=0} = 0$ satisfies Condition \eqref{poincare_ci}. Then, for all $t \in \R$,

\begin{equation*}
\zeta(t, \cdot) = \partial_{x} v(t, \cdot).
\end{equation*}
\end{lemma}

\noindent We also have the following dispersion result (see for instance \cite{Wp_estim_klein_gordon} and \cite{Lp_estim_klein_gordon} or Corollary 7.2.4 in \cite{hormander_nonlinear}).

\begin{lemma}\label{disp_kl}
\noindent Let $u_{0} \in W^{2,1}(\R)$. Then

\begin{equation*}
\lver \int_{\R} e^{ix \xi \pm  t \sqrt{\xi^{2} + 1}} u_{0}(\xi) d \xi \rver_{L^{\infty}_{x}} \leq \frac{C}{\sqrt{1+ |t|}} \lver u_{0} \rver_{W^{2,1}}.
\end{equation*}

\end{lemma}

\noindent We can give the main result of this part.

\begin{thm}
\noindent \noindent Let $\mu_{0} > 0$, \upshape $\zeta_{0}, u_{0}, v_{0}, \textbf{V}^{\sharp}_{0} \in H^{6}(\R)$\itshape, $x\zeta_{0}, xu_{0}, xv_{0} \in H^{4}(\R)$, such that  \upshape $\zeta_{0},v_{0}$ \itshape satisfy Condition \eqref{poincare_ci}, $1+\epsilon \zeta \geq h_{\min} > 0$ and \upshape$0 < \mu < \mu_{0}$ \itshape. Then, there exists a time $T > 0$, such that  there exists

\medskip

\noindent (i) a unique classical solution \upshape$\left( \zeta_{B}, u_{B}, v_{B}, \textbf{V}^{\sharp}_{B} \right)$ \itshape of \eqref{boussinesq_eq_poin} with initial data \upshape$\left(\zeta_{0}, u_{0}, v_{0}, \textbf{V}^{\sharp}_{0} \right)$ \itshape on $\left[ 0, \frac{T}{\sqrt{\mu}} \right]$.

\medskip

\noindent (ii)  a unique solution \upshape$\left(\zeta, u, v \right)$ \itshape of \eqref{boussines_linear} with initial data \upshape$\left(\zeta_{0}, u_{0}, v_{0} \right)$ \itshape on $\left[ 0, \frac{T}{\sqrt{\mu}} \right]$.

\medskip 

\noindent Moreover, we have the following error estimate for all $0 \leq t \leq \frac{T}{\sqrt{\mu}}$,

\upshape
\begin{equation*}
\lver \left(\zeta_{B},u_{B}, v_{B} \right) - \left(\zeta,u, v \right) \rver_{L^{\infty}([0,t] \times \R)} \leq C \left( \frac{\mu t}{1+ \sqrt{t}} + \mu^{2} t^{2} + \mu^{\frac{3}{2}} t \right) \leq C \mu^{\frac{3}{4}}.
\end{equation*}
\itshape

\noindent where \upshape $C = C\left(T, \frac{1}{h_{\min}}, \mu_{0}, \lver \zeta_{0} \rver_{H^{6}}, \lver u_{0} \rver_{H^{6}}, \lver v_{0} \rver_{H^{6}}, \lver \textbf{V}^{\sharp}_{0} \rver_{H^{6}}, \lver x \zeta_{0} \rver_{H^{4}} , \lver x u_{0} \rver_{H^{4}}, \lver x v_{0} \rver_{H^{4}} \right)$ \itshape.

\end{thm}

\begin{remark}
\noindent By standard energy estimates, we easily get that, for all $0 \leq t \leq \frac{T}{\sqrt{\mu}}$,

\upshape
\begin{align*}
\lver \left(\zeta_{B},u_{B}, v_{B} \right) - \left(\zeta,u, v \right) \rver_{L^{\infty}([0,t] \times \R)} \leq C \mu t \leq C \sqrt{\mu},
\end{align*}
\itshape

\noindent where $C$ is as in the previous theorem. Therefore, our result is not a simple energy estimate. We use the dispersive effects due to the Coriolis forcing to be more accurate.
\end{remark}

\begin{proof}
\noindent The first point follows from Proposition \ref{existence_boussi}. For the error estimate, if we denote by $U = \left(\zeta_{B},u_{B}, v_{B} \right)^{t}$, $U$ satisfies the linear Boussinesq-Coriolis equation up to a remainder of order $\mu$ and a remainder of order $\mu^{\frac{3}{2}}$. Then, using the Duhamel's formula we get

\begin{equation*}
U(t) = \mathcal{S}(t) U_{0} + \mu \int_{0}^{t} \mathcal{S}(t-\tau) \begin{pmatrix} - \partial_{x} \left( \zeta_{B} u_{B} \right)(\tau) \\ - u_{B}(\tau) \partial_{x} u_{B}(\tau) + \frac{1}{3} \partial_{x}^{2} \partial_{\tau} u_{B}(\tau) \\ - u_{B} \partial_{x} v_{B} \end{pmatrix} + \mu^{\frac{3}{2}} \int_{0}^{t} \hspace{-0.3cm} \mathcal{S}(t-\tau) R
\end{equation*}

\noindent where $R$ is a remainder bounded uniformly with respect to $\mu$. Then, using again the Duhamel's formula on the first integral we get

\begin{equation*}
\begin{aligned}
U(t) = \mathcal{S}(t) U_{0} &- \mu \int_{0}^{t} \mathcal{S}(t-\tau) \begin{pmatrix} \partial_{x} \left( (\mathcal{S}_{1}(\tau) U_{0}) (\mathcal{S}_{2}(\tau) U_{0}) \right) \\ (\mathcal{S}_{2}(\tau) U_{0}) \partial_{x} (\mathcal{S}_{2}(\tau) U_{0})  \\ (\mathcal{S}_{2}(\tau) U_{0}) \partial_{x} ( \mathcal{S}_{3}(\tau) U_{0} ) \end{pmatrix} \\
& + \mu \int_{0}^{t} \mathcal{S}(t-\tau) \begin{pmatrix} 0 \\  \frac{1}{3} \partial_{x}^{2} \partial_{\tau} \mathcal{S}_{2}(\tau) U_{0} \\ 0 \end{pmatrix} + \mu^{2} \int_{0}^{t} \int_{0}^{\tau} \tilde{R} + \mu^{\frac{3}{2}} \hspace{-0.1cm} \int_{0}^{t} \hspace{-0.2cm} \mathcal{S}(t-\tau) \tilde{R}\\
&\hspace{-1.6cm} = \mathcal{S}(t) U_{0} - \mu I_{1}(t) + \mu I_{2}(t) + \mu^{2} I_{3}(t) + \mu^{\frac{3}{2}} I_{4}(t),
\end{aligned}
\end{equation*}

\noindent where $\mathcal{S}_{i}(t)$ is the $i$th row of $\mathcal{S}(t)$. We start by estimating $I_{1}$. We have

\begin{equation*}
I_{1}(t) =\int_{0}^{t} \mathcal{S}(t-\tau) \begin{pmatrix} \partial_{x} \left( \zeta(\tau) u(\tau) \right) \\ u(\tau) \partial_{x} u(\tau)  \\  u(\tau) \partial_{x} v(\tau) \end{pmatrix}.
\end{equation*}

\noindent Then, we notice that $\partial_{x} \left( \zeta(\tau) u(\tau) \right) = \partial_{x} \left(u(\tau) \partial_{x} v(\tau)\right)$ since $\zeta(\tau) = \partial_{x} v(\tau)$ by Lemma \ref{propa_condition_poincare}. Therefore, using Lemma \ref{disp_kl} and products estimates, we get

\begin{align*}
\lver I_{1}(t) \rver_{L^{\infty}} &\leq \int_{0}^{t} \frac{1}{\sqrt{1+t-\tau}} \lver  \begin{pmatrix} \partial_{x} \left( (\mathcal{S}_{1}(\tau) U_{0}) (\mathcal{S}_{2}(\tau) U_{0}) \right) \\ (\mathcal{S}_{2}(\tau) U_{0}) \partial_{x} (\mathcal{S}_{2}(\tau) U_{0})  \\ (\mathcal{S}_{2}(\tau) U_{0}) \partial_{x} ( \mathcal{S}_{3}(\tau) U_{0} ) \end{pmatrix} \rver_{W^{2,1}}\\
&\leq C \left( \lver \zeta_{0} \rver_{H^{3}} , \lver u_{0} \rver_{H^{3}}, \lver v_{0} \rver_{H^{3}}, \lver \textbf{V}^{\sharp}_{0} \rver_{H^{3}} \right) \frac{t}{\sqrt{1+t}}.
\end{align*}

\noindent For $I_{2}$, using Lemma \ref{disp_kl} we get

\begin{equation*}
\lver I_{2} \rver \leq C \left( \lver \zeta_{0} \rver_{H^{4}} , \lver u_{0} \rver_{H^{4}}, \lver v_{0} \rver_{H^{4}}, \lver x \zeta_{0} \rver_{H^{4}} , \lver x u_{0} \rver_{H^{4}}, \lver x v_{0} \rver_{H^{4}} \right) \frac{t}{\sqrt{1+t}}.
\end{equation*}

\noindent Finally, using Proposition \ref{existence_boussi}, we have

\begin{equation*}
\begin{aligned}
&\lver I_{3}(t) \rver_{H^{1}} \leq C \left( \lver \zeta_{0} \rver_{H^{6}} , \lver u_{0} \rver_{H^{6}}, \lver v_{0} \rver_{H^{6}}, \lver \textbf{V}^{\sharp}_{0} \rver_{H^{6}} \right) t^{2}\\
&\lver I_{4}(t) \rver_{H^{1}} \leq C \left( \lver \zeta_{0} \rver_{H^{4}} , \lver u_{0} \rver_{H^{4}}, \lver v_{0} \rver_{H^{4}}, \lver \textbf{V}^{\sharp}_{0} \rver_{H^{4}} \right) t.
\end{aligned}
\end{equation*}

\noindent Gathering these four estimates, we get the result.

\end{proof}

\noindent Hence, using Theorem \ref{boussi_water_wave_compare}, we justify that poincar\'e waves remain close to the solutions of the water waves equations \eqref{castro_lannes_formulation} over times $\mathcal{O}_{\mu} \left(1 \right)$ with an accuracy of order $\mathcal{O} \left( \mu \right)$. Furthermore, if one can show that a solution of the water waves equations \eqref{castro_lannes_formulation}, with initial data satisfying Condition \eqref{poincare_ci}, exists over a time $\mathcal{O} \left(\frac{1}{\sqrt{\mu}} \right)$, we show that this solution remains close, with an accuracy of order $\mathcal{O} \left( \mu^{\frac{3}{4}} \right)$, to the solution of the linear Boussinesq-Coriolis equations with the same initial data. The reader interested in more linear properties of the water waves equations in shallow water can refer to Chapter 4 in \cite{Majda_ocean}.

\subsection{Weak rotation, the Ostrovsky equation}\label{model_ostrov}

\noindent Without Coriolis forcing and vorticity, it is well-known, that the KdV equation is a good approximation of the water waves equation under the assumption that $\epsilon$ and $\mu$ have the same order (\cite{craig_boussinesq}, \cite{kano_nishida}, \cite{schneider_wayne_longwave}, \cite{bona_colin_lannes}, Part 7.1 in \cite{Lannes_ww}). When the Coriolis forcing is taken into account, Ostrovsky (\cite{ostrovsky_original}) derived an equation for long waves, which is an adaptation of the KdV equation,

\begin{equation}\label{ostrov_eq}
\partial_{\xi} \left( \partial_{\tau} k +  \frac{3}{2} k \partial_{\xi} k + \frac{1}{6} \partial_{\xi}^{3} k \right) = \frac{1}{2} k.
\end{equation}

\noindent This equation is called the Ostrovsky equation or rKdV-equation in the physical literature. Initially developed for internal water waves, several authors also studied it for surface water waves (\cite{ostrovsky_surf}, \cite{germain-renouard}, \cite{leonov}, \cite{grimshaw_ostrov_surf}). The purpose of this part is to fully justify it. Inspired by \cite{germain-renouard} we consider the asymptotic regime

\begin{equation}\label{ostrov_regime}
\mathcal{A}_{\text{Ost}} = \left\{ \left( \epsilon, \beta, \gamma, \mu, \text{Ro} \right), 0 \leq \mu \leq \mu_{0}, \epsilon = \mu, \beta = \gamma = 0, \frac{\epsilon}{\text{Ro}} = \sqrt{\mu} \right\}.
\end{equation}

\noindent Then, the Boussinesq-Coriolis equations become (see Remark \ref{weak_rot_boussi})

\begin{equation}\label{boussi_weakrot}
\left\{
\begin{aligned}
&\partial_{t} \zeta + \partial_{x} \left( [1+\mu \zeta] u  \right) = 0,\\
&\left(1- \frac{\mu}{3} \partial_{x}^{2} \right) \partial_{t} u + \partial_{x} \zeta + \mu u \partial_{x} u - \sqrt{\mu} v = 0,\\
&\partial_{t} v + \mu u \partial_{x} v + \sqrt{\mu} u = 0.
\end{aligned}
\right.
\end{equation}

\noindent In order to motivate our approach, let us recall that we are interested in the one-dimensional propagation of water waves in the long wave regime. If we drop all the terms of order $\mathcal{O}(\sqrt{\mu})$ in the Boussinesq-Coriolis, we obtain that

\begin{equation*}
\left\{
\begin{aligned}
&\partial_{t} \zeta + \partial_{x} u = 0,\\
&\partial_{t} u + \partial_{x} \zeta = 0,\\
&\partial_{t} v = 0.
\end{aligned}
\right.
\end{equation*}

\noindent Hence, if we assume that $v$ is initially zero, we get a wave equation and propagation of traveling water waves with speed $\pm 1$. Therefore, it is natural to study how these  traveling water waves are perturbed when we add weakly nonlinear effects, i.e when we consider the System \eqref{boussi_weakrot}. In this paper, we consider only water waves with speed $1$. We consider a WKB expansion for $\left(\zeta,u,v\right)$. We seek an approximate solution $\left(\zeta_{app}, u_{app}, v_{app} \right)$ of \eqref{boussi_weakrot} under the form

\begin{equation}\label{ansatz_ostrov}
\begin{aligned}
&\zeta_{app}(t,x) = k(x-t,\mu t) + \mu \zeta_{(1)}(t,x,\mu t),\\
&u_{app}(t,x) = k(x-t,\mu t) + \mu u_{(1)}(t,x,\mu t),\\
&v_{app}(t,x) = \sqrt{\mu} v_{(1/2)}(t,x,\mu t).
\end{aligned}
\end{equation}

\noindent where $k = k(\xi,\tau)$ is our modulated traveling water waves, and the others terms are correctors. Then, we plug the ansatz in Sytem \eqref{boussi_weakrot} and we get

\begin{equation}\label{boussi_weak_app}
\begin{aligned}
&\partial_{t} \zeta_{app} + \partial_{x} \left( [1+\mu \zeta_{app}] u_{app}  \right) = \mu R_{(1)}^{1} + \mu^{2} R_{1},\\
&\left(1- \frac{\mu}{3} \partial_{x}^{2} \right) \partial_{t} u_{app} + \partial_{x} \zeta_{app} + \mu u_{app} \partial_{x} u_{app} - \sqrt{\mu} v_{app} = \mu R_{(1)}^{2} + \mu^{2} R_{2},\\
&\partial_{t} v_{app} + \mu u_{app} \partial_{x} v_{app} + \sqrt{\mu} u_{app} = \sqrt{\mu} R_{(1/2)}^{3} + \mu^{\frac{3}{2}} R_{3},
\end{aligned}
\end{equation}

\noindent where

\begin{equation*}
\begin{aligned}
&R_{(1)}^{1} = \partial_{t} \zeta_{(1)} + \partial_{x} u_{(1)} + \partial_{\tau} k + 2 k \partial_{\xi} k,\\
&R_{(1)}^{2} = \partial_{t} u_{(1)} + \partial_{x} \zeta_{(1)} + \partial_{\tau} k +  \frac{1}{3} \partial^{3}_{\xi}  k + k \partial_{\xi} k - v_{(1/2)},\\
&R_{(1/2)}^{3} = \partial_{t} v_{(1/2)} + k,
\end{aligned}
\end{equation*}

\noindent and 

\begin{equation}\label{remainder}
\begin{aligned}
&R_{1} = \partial_{\tau} \zeta_{(1)}  + \partial_{x} \left(k u_{(1)} + k \zeta_{(1)} + \mu \zeta_{(1)} u_{(1)} \right),\\
&R_{2} = \partial_{\tau} u_{(1)} - \frac{1}{3} \partial_{\xi}^{3} \partial_{\tau} k - \frac{1}{3} \partial_{x}^{3} \partial_{t} u_{(1)}  - \mu \frac{1}{3} \partial_{x}^{3} \partial_{\tau} u_{(1)} + \partial_{x} \left(k u_{(1)} \right) + \mu u_{(1)} \partial_{x} u_{(1)},\\
&R_{3} = \partial_{\tau} v_{(1/2)} + \left(k + \sqrt{\mu} u_{(1)} \right) \partial_{x} v_{(1/2)} + u_{(1)}.
\end{aligned}
\end{equation}

\noindent Then, the idea is to choose the correctors with $R_{(1)}^{1}(t,x,\tau) = R_{(1)}^{2}(t,x,\tau) = 0$ and $R_{(1/2)}^{3}(t,x,\tau) = 0$ for all $x \in \R$, $t \in \left[0, \frac{T}{\mu} \right]$ and $\tau \in \left[0,T \right]$.

\begin{remark}
\noindent In fact, we should add $\sqrt{\mu} \zeta_{(1/2)}(t,x,\mu t)$, $\sqrt{\mu} u_{(1/2)}(t,x,\mu t)$, $v_{(0)}(t,x,\mu t)$, and $\mu v_{(1)}(t,x,\mu t)$ to the ansatz \eqref{ansatz_ostrov} for $\zeta_{app}$, $u_{app}$, $v_{app}$ and $v_{app}$ respectively. However, if we plug them in System \eqref{boussi_weakrot} and we want to cancel all the terms of order $\sqrt{\mu}$ and $\mu$, we get 

\begin{equation*}
\begin{aligned}
&\partial_{t} \zeta_{(1/2)} + \partial_{x} u_{(1/2)} = 0,\\
&\partial_{t} u_{(1/2)} + \partial_{x} \zeta_{(1/2)} + v_{(0)} = 0,\\
&\partial_{t} v_{(0)} = 0,\\
&\partial_{t} v_{(1)} + \partial_{\tau} v_{(0)} + k \partial_{x} v_{(0)} + u_{(1/2)} = 0,
\end{aligned}
\end{equation*}

\noindent which leads to $\zeta_{(1/2)} = u_{(1/2)} = v_{(0)} = v_{(1)} = 0$ if these quantities are initially zero. Hence, we make this assumption in the following.
\end{remark}

\noindent Then, if we assume that $v_{(1/2)}$ and $k$ vanish at $x=\infty$, the condition $R_{(1/2)}^{3} = 0$ is equivalent to the equation

\begin{equation*}
\partial_{t} \partial_{x} v_{(1/2)}(t,x,\tau) + \partial_{\xi} k(x-t,\tau) = 0.
\end{equation*}

\noindent Since, $\partial_{t} (k(x-t,\tau)) = - \partial_{\xi} k(x-t,\tau)$, we can take

\begin{equation}\label{eq_v_f}
\partial_{x} v_{(1/2)}(t,x,\tau) = \partial_{x} v^{0}_{(1/2)}(x) - k^{0}(x) + k(x-t,\tau),
\end{equation}

\noindent where $v^{0}_{(1/2)}$ and $k^{0}$ are the initial data of $v_{(1/2)}$ and $k$ respectively. Then, we have to introduce the following spaces.

\begin{definition}\label{partialx_Hs}
\noindent For $s \in \R$, we define the Hilbert spaces $\partial_{x} H^{s}(\R)$ as

\begin{equation*}
\partial_{x} H^{s}(\R) =  \left\{ k \in H^{s-1}(\R) \text{,  } k = \partial_{x} \tilde{k} \text{  with  } \tilde{k} \in H^{s}(\R)\right\},
\end{equation*}

\noindent and $\tilde{k}$ is denoted $\partial_{x}^{-1} k$ in the following. In the same way, we define $\partial_{x}^{2} H^{s}(\R)$. 
\end{definition}

\noindent Then, if we assume that $k(\cdot,\tau) \in \partial_{x} H^{s}(\R)$ for all $\tau \in [0,T]$, we have 

\begin{equation*}
v_{(1/2)}(t,x,\tau) = v^{0}_{(1/2)}(x) - \partial_{x}^{-1} k^{0}(x) + \partial_{x}^{-1} k(x-t,\tau),
\end{equation*}

\noindent Furthermore, from $R_{(1)}^{1} = R_{(1)}^{2} = 0$, if we denote $w_{\pm} = \zeta_{(1)} \pm u_{(1)}$ we get

\begin{equation}\label{wave_ostrov}
\begin{aligned}
&\left(\partial_{t} + \partial_{x} \right) w_{+} + \left( 2 \partial_{\tau} k + 3 k \partial_{\xi} k + \frac{1}{3} \partial_{\xi}^{3} k - \partial_{\xi}^{-1} k \right) (x-t,\tau) - \left(v^{0}_{(1/2)} - \partial_{\xi}^{-1} k^{0}\right)(x) = 0,\\
&\left(\partial_{t} - \partial_{x} \right) w_{-} + \left(k \partial_{\xi} k - \frac{1}{3} \partial_{\xi}^{3} k + \partial_{\xi}^{-1} k \right) (x-t,\tau) + \left( v^{0}_{(1/2)} - \partial_{\xi}^{-1} k^{0}\right)(x) = 0.\\
\end{aligned}
\end{equation}

\noindent The following lemma (Lemma 7.6 in \cite{Lannes_ww}) is the key point to control $u$ and $v$.

\begin{lemma}\label{control_diff_speed}
\noindent Let $c_{1} \neq c_{2}$. Let $k_{1}, k_{2}, k_{3} \in L^{2}(\R)$ with $k_{2} = K'_{2}$ and $K_{2} \in L^{2}(\R)$. We consider the unique solution $k$ of

\begin{equation*}
\left\{
\begin{aligned}
&(\partial_{t} + c_{1} \partial_{x} )k = k_{1}(x-c_{1} t) + k_{2}(x-c_{2} t) + k_{3}(x-c_{2} t),\\
&k_{|t=0} = 0.
\end{aligned}
\right.
\end{equation*}

\noindent Then, $\underset{t \shortrightarrow \infty}{\lim} \lver \frac{1}{t} k(t, \cdot) \rver_{2} = 0$ if and only if $k_{1} \equiv 0$ and in that case

\begin{equation*}
\lver k(t, \cdot) \rver_{2} \leq \frac{C}{\lver c_{1} - c_{2} \rver} \left(\lver K_{2} \rver_{2} \frac{t}{1+t} + \lver k_{3} \rver_{H^{2}} \frac{t}{1+\sqrt{t}}\right).
\end{equation*}
\end{lemma}

\noindent Then, in order to avoid a linear growth for the solution of \eqref{wave_ostrov}, we also have to impose that 

\begin{equation}\label{ostrov_eq_f}
\partial_{\tau} k + \frac{3}{2} k \partial_{\xi} k + \frac{1}{6} \partial_{\xi}^{3} k = \frac{1}{2} \partial_{\xi}^{-1} k,
\end{equation}

\noindent which is the Ostrovsky equation. Before giving a full justification of the Ostrovsky equation, we need a local wellposedness result of this equation. The following proposition is a generalization of Theorem 2.1 in \cite{linares_ostrov} and Theorem 2.6 in \cite{varlamov_ostrov} (see also \cite{ostrov_low_reg} for weak solutions).

\begin{prop}\label{existence_ostrov}
\noindent Let $s > \frac{7}{4}$ and $k_{0} \in \partial_{x} H^{s}(\R)$. Then, there exists a time $T > 0$ and a unique solution $k \in \mathcal{C} \left([0,T]; \partial_{x} H^{s}(\R)) \right)$ to the Ostrovsky equation \eqref{ostrov_eq_f} and one has

\begin{equation*}
\lver \partial_{\xi}^{-1} k(t,\cdot) \rver_{H^{s}} \leq C \left(T, \lver \partial_{\xi}^{-1} k_{0} \rver_{H^{s}} \right).
\end{equation*}

\noindent Moreover, if $s \geq 3$, $k_{0} \in \partial_{x}^{2} H^{s+1}(\R)$, $k \in \mathcal{C} \left([0,T]; \partial_{x}^{2} H^{s+1}(\R)) \right)$ and one has

\begin{equation*}
\lver \partial_{\xi}^{-2} k(t,\cdot) \rver_{H^{s+1}} \leq C \left(T, \lver \partial_{\xi}^{-2} k_{0} \rver_{H^{s+1}} \right).
\end{equation*}

\end{prop}

\begin{proof}
\noindent We only prove the second point of the Proposition. We denote by $S(t)$ the semi-group of the linearized Ostrovsky equation

\begin{equation*}
\partial_{\tau} k + \frac{1}{6} \partial_{\xi}^{3} k - \frac{1}{2} \partial_{\xi}^{-1} k = 0,
\end{equation*}

\noindent and it is easy to check that this semi-group acts unitary on $H^{s}(\R)$. We denote $\tilde{k} = \partial_{\tau} k$. Then, $\tilde{k}$ satisfies the equation

\begin{equation*}
\partial_{\tau} \tilde{k} + \frac{3}{2} \partial_{\xi} \left(\tilde{k} k \right) + \frac{1}{6} \partial_{\xi}^{3} \tilde{k} - \frac{1}{2} \partial_{\xi}^{-1} \tilde{k} = 0.
\end{equation*}

\noindent Using the Duhamel's formula we obtain

\begin{equation*}
\partial_{\xi}^{-1} \tilde{k}(t, \cdot) = S(t) \partial_{\xi}^{-1} \tilde{k}_{0} + \frac{3}{2} \int_{0}^{t} S(t-s) \left( k \tilde{k} \right)(s,\cdot) ds.
\end{equation*}

\noindent Since $\partial_{\xi}^{-1} \tilde{k}_{0} = - \frac{3}{4} k_{0}^{2} - \partial_{\xi}^{2} k_{0} + \partial_{\xi}^{-2} k_{0} \in L^{2}(\R)$, we get the result since we have

\begin{equation*}
\frac{1}{2} \partial_{\xi}^{-2} k = \partial_{\xi}^{-1} \partial_{\tau} k + \frac{3}{4} k^{2} + \frac{1}{6} \partial_{\xi}^{2} k.
\end{equation*}
\end{proof}

\noindent Notice that contrary to the KdV equation, we can not expect a global existence. We can now give the main result of this part.

\begin{thm}\label{full_just_ostrov}
\noindent Let $k^{0} \in \partial_{x}^{2} H^{10}(\R)$, such that $1+\epsilon k^{0} \geq h_{\min} > 0$, $v^{0} \in \partial_{x} H^{6}(\R)$ and $\mu_{0} > 0$. Then, there exists a time $T > 0$, such that for all $0 < \mu \leq \mu_{0}$, we have

\medskip

\noindent (i) a unique classical solution \upshape$\left(\zeta_{B},u_{B}, v_{B} \right)$ \itshape of \eqref{boussi_weakrot} with initial data \upshape$\left(k^{0}, k^{0}, \sqrt{\mu} v^{0} \right)$ \itshape on $\left[ 0, \frac{T}{\mu} \right]$.

\medskip

\noindent (ii)  a unique classical solution \upshape$k$ \itshape of \eqref{ostrov_eq_f} with initial data \upshape$k^{0}$ \itshape on $\left[ 0, T \right]$.

\medskip 

\noindent (iii) If we define $\left(\zeta_{Ost}, u_{Ostr} \right)(t,x) = \left( k(x-t,\mu t),k(x-t,\mu t ) \right)$ we have the following error estimate for all $0 \leq t \leq \frac{T}{\mu}$,

\begin{equation*}
\lver \left(\zeta_{B},u_{B} \right) - \left(\zeta_{Ost},u_{Ost} \right) \rver_{L^{\infty}([0,t] \times \R)} \leq  C \left((1+\sqrt{\mu} t) \frac{\mu t}{1+t} + \mu^{\frac{3}{2}} t \right)
\end{equation*}
\itshape

\noindent where $C = C\left(T, \frac{1}{h_{\min}}, \mu_{0}, \lver \partial_{x}^{-2} k^{0} \rver_{H^{10}}, \lver \partial_{x}^{-1} v^{0} \rver_{H^{6}} \right)$.

\end{thm}

\begin{proof}
\noindent In all the proof, $C$ will be a constant as in the theorem. The first and second point follow from Corollary \ref{existence_weak_rot_boussi} and \ref{existence_ostrov}. In order to get the error estimate, we have to control the remainders $R_{1}, R_{2}, R_{3}$,  defined in \eqref{remainder}. First, using Lemma \ref{control_diff_speed},  the fact that we can express the quantities $\frac{1}{2} \partial_{\xi} k^{2} - \frac{1}{3} \partial_{\xi}^{3} k$, $\partial_{\xi}^{-1} k$ and $v_{0}$ as derivatives with respect to $x$ and the fact that $k$ satisfies the Ostrovsky equation \eqref{ostrov_eq_f}, we have 

\begin{equation*}
\lver \zeta_{(1)} \rver_{2} + \lver u_{(1)} \rver_{2} \leq C \frac{t}{1 + t}.
\end{equation*} 

\noindent But we can also control all the derivatives with respect to $\tau$ or $x$ of $u$ and $v$ be differentiating \eqref{wave_ostrov}. Hence, we get a control for the remainders $R_{1}$ and $R_{2}$. For $R_{3}$, we use the fact that $v= \partial_{x}^{-1} k$. We finally, obtain

\begin{equation*}
\lver R_{1} \rver_{H^{2}} + \lver R_{2} \rver_{H^{2}} + \lver R_{3} \rver_{H^{2}} \leq C \left( \frac{t}{1+t} + \mu t + 1\right),
\end{equation*}
 
\noindent Then, thanks to Proposition \ref{stability_boussi} and remark \ref{weak_rot_boussi}, we get

\begin{equation*}
\lver \left(\zeta_{B}, u_{B}, v_{B} \right) - \left(\zeta_{app}, u_{app}, v_{app} \right) \rver_{L^{\infty}([0,t] \times \R)} \leq C \mu^{\frac{3}{2}} t \left( \frac{t}{1+t} + \mu t +1 \right).
\end{equation*}

\noindent Moreover, we have

\begin{equation*}
\lver \left(\zeta_{app}, u_{app} \right) - \left(\zeta_{Ost},v_{Ost} \right) \rver_{L^{\infty}([0,t] \times \R)} \leq \mu \frac{t}{1+t}.
\end{equation*}

\noindent Then, the result follows easily.

\end{proof}

\noindent This theorem, combined with Theorem \ref{boussi_water_wave_compare}, shows that the solutions of the water waves equations \eqref{castro_lannes_formulation} is well approximated over times $\mathcal{O} \left(\frac{1}{\sqrt{\mu}} \right)$ with an accuracy of order $\mathcal{O} \left( \mu \right)$ by the Ostrovsky approximation if we have a small Coriolis forcing. The approach we develop here is similar to the one of the KP equations (see for instance \cite{Saut_Lannes_KP}, \cite{Alvarez_Lannes} or Part 7.2.1 in \cite{Lannes_ww}). The fact that $k^{0} \in \partial_{x} H^{8}$ is essential and physical since a solution of the Ostrovsky equation has to be mean free. However, we suppose here that $k^{0} \in \partial_{x}^{2} H^{9}(\R)$ and  $v^{0} \in \partial_{x} H^{5}(\R)$ which is more restrictive. In fact, using the strategy developed in \cite{benyoussef_lannes} for the KP approximation we can hope to release this assumption. Finally, notice that contrary to the KdV equation, the Ostrovsky equation does not admit solitons (\cite{zhang_liu_no_soliton_ostrov}, \cite{galkin_no_soliton_ostrov}).

\subsection{Very weak rotation, the KdV equation}\label{model_kdv}

\noindent As we said before, without Coriolis forcing, it is well-known, that the KdV equation is a good approximation of the water waves equations. In this part we show that if $\frac{\epsilon}{\text{Ro}}$ is small enough, we get the KdV equation as an asymptotic model. We recall the KdV equation

\begin{equation}\label{kdv_eq}
\partial_{\tau} k +  \frac{3}{2} k \partial_{\xi} k + \frac{1}{6} \partial_{\xi}^{3} k = 0.
\end{equation}

\noindent Inspired by \cite{germain-renouard}, we show that $\frac{\epsilon}{\text{Ro}} = \mathcal{O}(\mu)$ is sufficient. we consider the asymptotic regime

\begin{equation}\label{kdv_regime}
\mathcal{A}_{KdV} = \left\{ \left( \epsilon, \beta, \gamma, \mu, \text{Ro} \right), 0 \leq \mu \leq \mu_{0}, \epsilon = \mu, \beta = \gamma = 0, \frac{\epsilon}{\text{Ro}} = \mu \right\}.
\end{equation}

\noindent Then, the Boussinesq-Coriolis equations become (see Remark \ref{weak_rot_boussi})

\begin{equation}\label{boussi_kdv}
\left\{
\begin{aligned}
&\partial_{t} \zeta + \partial_{x} \left( [1+\mu \zeta] u  \right) = 0,\\
&\left(1- \frac{\mu}{3} \partial_{x}^{2} \right) \partial_{t} u + \partial_{x} \zeta + \mu u \partial_{x} u - \mu v = 0,\\
&\partial_{t} v + \mu u \partial_{x} v + \mu u = 0.
\end{aligned}
\right.
\end{equation}

\noindent Proceeding as in the previous part, we seek an approximate solution $\left(\zeta_{app}, u_{app}, v_{app} \right)$ of \eqref{boussi_kdv} under the form

\begin{equation}\label{ansatz_kdv}
\begin{aligned}
&\zeta_{app}(t,x) = k(x-t,\mu t) + \mu \zeta_{(1)}(t,x,\mu t),\\
&u_{app}(t,x) = k(x-t,\mu t) + \mu u_{(1)}(t,x,\mu t),\\
&v_{app}(t,x) = \mu v_{(1)}(t,x,\mu t).
\end{aligned}
\end{equation}

\noindent Then, we plug the ansatz in Sytem \eqref{boussi_kdv} and we get

\begin{equation}\label{boussi_kdv_app}
\begin{aligned}
&\partial_{t} \zeta_{app} + \partial_{x} \left( [1+\mu \zeta_{app}] u_{app}  \right) = \mu R_{(1)}^{1} + \mu^{2} R_{1},\\
&\left(1- \frac{\mu}{3} \partial_{x}^{2} \right) \partial_{t} u_{app} + \partial_{x} \zeta_{app} + \mu u_{app} \partial_{x} u_{app} - \mu v_{app} = \mu R_{(1)}^{2} + \mu^{2} R_{2},\\
&\partial_{t} v_{app} + \mu u_{app} \partial_{x} v_{app} + \mu u_{app} = \mu R_{(1)}^{3} + \mu^{2} R_{3},
\end{aligned}
\end{equation}

\noindent where

\begin{equation*}
\begin{aligned}
&R_{(1)}^{1} = \partial_{t} \zeta_{(1)} + \partial_{x} u_{(1)} + \partial_{\tau} k + 2 k \partial_{\xi} k,\\
&R_{(1)}^{2} = \partial_{t} u_{(1)} + \partial_{x} \zeta_{(1)} + \partial_{\tau} k +  \frac{1}{3} \partial^{3}_{\xi}  k + k \partial_{\xi} k,\\
&R_{(1)}^{3} = \partial_{t} v_{(1)} + k,
\end{aligned}
\end{equation*}

\noindent and

\begin{equation*}
\begin{aligned}
&R_{1} = \partial_{\tau} \zeta_{(1)} + \partial_{x} \left(k u_{(1)} + k \zeta_{(1)} + \mu \zeta_{(1)} u_{(1)} \right),\\
&R_{2} = \partial_{\tau} u_{(1)} - \frac{1}{3} \partial_{\xi}^{3} \partial_{\tau} k - \frac{1}{3} \partial_{x}^{3} \partial_{t} u_{(1)}  - \mu \frac{1}{3} \partial_{x}^{3} \partial_{\tau} u_{(1)} + \partial_{x} \left(k u_{(1)} \right) + \mu u_{(1)} \partial_{x} u_{(1)} - v_{(1)},\\
&R_{3} = \partial_{\tau} v_{(1)} + \mu \left(k + \mu u_{(1)} \right) \partial_{x} v_{(1)} + u_{(1)}.
\end{aligned}
\end{equation*}

\begin{remark}
\noindent We should also add $v_{(0)}(t,x,\mu t)$ to the ansatz \eqref{ansatz_kdv} for $v_{app}$. However, if we plug it in System \eqref{boussi_kdv} we get $\partial_{t} v_{(0)} = 0$ which leads to $v_{(0)} = 0$ if the quantity is initially zero. Hence, we make this assumption in the following.
\end{remark}

\noindent As before, we assume that $R_{(1)}^{1}(t,x,\tau) = R_{(1)}^{2}(t,x,\tau) = R_{(1)}^{3}(t,x,\tau) = 0$ for all $x \in \R$, $t \in \left[0, \frac{T}{\mu} \right]$ and $\tau \in \left[0,T \right]$ which leads to $v_{(1)} = v^{0}_{(1)} - \partial_{x}^{-1} k^{0} + \partial_{x}^{-1} k$ and, if we denote $w_{\pm} = \zeta_{(1)} \pm u_{(1)}$ we get

\begin{equation*}
\begin{aligned}
&\left(\partial_{t} + \partial_{x} \right) w_{+} + \left( 2 \partial_{\tau} k + 3 k \partial_{\xi} k + \frac{1}{3} \partial_{\xi}^{3} k \right) (x-t,\tau) = 0,\\
&\left(\partial_{t} - \partial_{x} \right) w_{-} + \left(k \partial_{\xi} k - \frac{1}{3} \partial_{\xi}^{3} k \right) (x-t,\tau) = 0\\
\end{aligned}
\end{equation*}

\noindent and to avoid a linear growth of $u$ or $v$ we need that $k$ satisfies \eqref{kdv_eq}. We also have an existence result for the KdV equation (see for instance \cite{kenig_ponce_vega_kdv}).

\begin{prop}\label{existence_kdv}
\noindent Let $s \geq 1$, $k_{0} \in H^{s}(\R)$ and $T > 0$. Then, there exists a unique solution to the KdV equation \eqref{kdv_eq} $k \in \mathcal{C} \left([0,T]; H^{s}(\R)) \right)$ and one have

\begin{equation*}
\lver k \rver_{H^{s}} \leq C \left(T, \lver k_{0} \rver_{H^{s}} \right).
\end{equation*}

\noindent Moreover, if $s \geq 2$ and $k_{0} \in \partial_{x} H^{s+1}(\R)$, $k \in \mathcal{C} \left([0,T]; \partial_{x} H^{s+1}(\R)) \right)$ and we have

\begin{equation*}
\lver \partial_{x}^{-1} k \rver_{H^{s+1}} \leq C \left(T, \lver \partial_{x}^{-1} k_{0} \rver_{H^{s+1}} \right).
\end{equation*}
\end{prop}

\noindent Then, proceeding as in the previous part, we obtain the following theorem.

\begin{thm}
\noindent \noindent Let $k^{0} \in \partial_{x} H^{9}(\R)$, such that such that $1+\epsilon k^{0} \geq h_{\min} > 0$, $v^{0} \in H^{5}(\R)$ and $\mu_{0} > 0$. Then, there exists a time $T > 0$, such that for all $0 < \mu \leq \mu_{0}$, we have

\medskip

\noindent (i) a unique classical solution \upshape$\left(\zeta_{B},u_{B}, v_{B} \right)$ \itshape of \eqref{boussi_kdv} with initial data \upshape$\left(k^{0}, k^{0}, \mu v^{0} \right)$ \itshape on $\left[ 0, \frac{T}{\mu} \right]$.

\medskip

\noindent (ii)  a unique classical solution \upshape$k$ \itshape of \eqref{kdv_eq} with initial data \upshape$k^{0}$ \itshape on $\left[ 0, T \right]$.

\medskip 

\noindent (iii) If we define $\left(\zeta_{KdV}, u_{KdV} \right)(t,x) = \left( k(x-t,\mu t),k(x-t,\mu t ) \right)$ we have the following error estimate for all $0 \leq t \leq \frac{T}{\mu}$,

\begin{equation*}
\lver \left(\zeta_{B},u_{B} \right) - \left(\zeta_{KdV},u_{KdV} \right) \rver_{L^{\infty}([0,t] \times \R)} \leq  C \left( \frac{\mu t}{1+t} + \mu^{2} t \right)
\end{equation*}
\itshape

\noindent where $C = C\left(T, \frac{1}{h_{\min}}, \mu_{0}, \lver \partial_{x}^{-1} k^{0} \rver_{H^{9}}, \lver v^{0} \rver_{H^{5}} \right)$.

\end{thm}

\noindent This theorem, combined with Theorem \ref{boussi_water_wave_compare}, shows that the solutions of the water waves equations \eqref{castro_lannes_formulation} is well approximated over times $\mathcal{O} \left(\frac{1}{\mu} \right)$ with an accuracy of order $\mathcal{O} \left( \mu \right)$ by the KdV approximation if we have a very small Coriolis forcing. Notice that contrary to the irrotational case, the transverse velocity $v$ is not zero (also noticed in  \cite{germain-renouard}). Furthermore, in our situation,  the initial data for the KdV equation has to be of zero mean which means that we can not expect the propagation of  solitons on a large time (they have a constant sign) if $\frac{\epsilon}{\text{Ro}}$ and $\mu$ have the same order.

\section{Green-Naghdi equations for $\gamma = 0$ and $\beta = \mathcal{O} \left( \mu \right)$}\label{derive_gn}

\noindent This part is devoted to the derivation and justification of the Green-Naghdi equations \eqref{green_naghdi} under a Coriolis forcing, with $\gamma = 0$ and for small amplitude topography variations ($\beta = \mathcal{O}(\mu)$). The Green-Naghdi equations are originally obtained in the irrotational framework under the assumption that $\mu$ is small (no assumption on $\epsilon$) and by neglecting all the terms of order $\mathcal{O}(\mu^{2})$ in the water waves equations (see for instance \cite{green_naghdi_uneven_bott} or Part 5.1.1.2 in \cite{Lannes_ww}). It is a system of two equations on the surface $\zeta$ and the averaged horizontal velocity $\overline{\textbf{V}}$. These equations were generalized in \cite{Castro_Lannes_shallow_water} in presence of vorticity but without a Coriolis forcing. This new system is a cascade of equations that involves a second order tensor and a third order tensor. After deriving these equations, we show that they are an order $\mathcal{O}(\mu^{2})$ approximation of the water waves equations. We consider the asymptotic regime for the 1D Green-Naghdi equations

\begin{equation}\label{GN_regime}
\mathcal{A}_{\text{GN}} = \left\{ \left( \epsilon, \beta, \gamma, \mu, \text{Ro} \right), 0 \leq \mu \leq \mu_{0}, 0 \leq \epsilon, \frac{\epsilon}{\text{Ro}} \leq 1, \beta = \mathcal{O} \left( \mu \right), \gamma = 0 \right\}.
\end{equation}

\noindent The next subsection is devoted to extending Proposition \ref{eq_Qx} and \ref{eq_Qy}.

\subsection{Improvements for the equations of $\text{Q}_{x}$ and $\text{Q}_{y}$}\label{improve_Q}

\noindent We start by extending Proposition \ref{eq_Qx}.

\begin{prop}\label{eq_Qx_improve}
If \upshape $\left(\zeta, \Umuszero, \bm{\omega} \right)$ \itshape satisfy the Castro-Lannes system \eqref{castro_lannes_formulation}, then $\text{Q}_{x}$ satisfies the following equation

\upshape
\begin{equation*}
\begin{aligned}
\partial_{t} Q_{x} + \epsilon \overline{u} \partial_{x} Q_{x} + \epsilon Q_{x} \partial_{x} \overline{u} + \frac{\epsilon}{\text{Ro} \sqrt{\mu}} \left(\underline{v} - \overline{v} \right) = & - \epsilon \sqrt{\mu} \frac{1}{h} \partial_{x} \int_{-1+\beta b}^{\epsilon \zeta} \left(\uastsh \right)^{2}\\
&+ \epsilon \sqrt{\mu} \text{Q}_{x} \partial_{x} \text{Q}_{x} + \epsilon \mu \frac{1}{3} \partial_{x} \left( h^{2} \text{Q}_{x} \partial_{x}^{2} \overline{u} \right)\\
&+ \epsilon \mu \frac{1}{6} h^{2} u^{\sharp} \partial_{x}^{3} \overline{u} + \epsilon \mu \frac{1}{8h} \partial_{x} \left( h^{3} u^{\sharp} \right) \partial_{x}^{2} \overline{u}\\
&+ \epsilon \max \left( \beta \sqrt{\mu}, \mu^{\frac{3}{2}} \right) R,
\end{aligned}
\end{equation*}
\itshape

\noindent and $\uastsh$ satisfies the equation

\upshape
\begin{equation*}
\begin{aligned}
\partial_{t} \uastsh + \epsilon \overline{u} \partial_{x} \uastsh + \epsilon \uastsh \partial_{x} \overline{u} + \frac{\epsilon}{\text{Ro} \sqrt{\mu}} \left(\overline{v} - v \right) = & \epsilon \sqrt{\mu} \frac{1}{h} \partial_{x} \int_{-1+\beta b}^{\epsilon \zeta} \left(\uastsh \right)^{2} - \epsilon \sqrt{\mu} \uastsh \partial_{x} \uastsh\\
& + \epsilon \partial_{x} \left(\int_{-1+\beta b}^{z} \left[\overline{u} + \sqrt{\mu} \uastsh \right] \right) \partial_{z} \uastsh\\ 
&+ \epsilon \mu R.
\end{aligned}
\end{equation*}
\itshape

\end{prop}

\begin{proof}
\noindent Using the second equation of the vorticity equation of the Castro-Lannes system \eqref{castro_lannes_formulation}, we have

\begin{equation*}
\partial_{t} \bm{\omega}_{y} + \epsilon u \partial_{x} \bm{\omega}_{y} + \frac{\epsilon}{\mu} \text{w} \partial_{z}  \bm{\omega}_{y} = \epsilon \bm{\omega}_{x} \partial_{x} v + \frac{\epsilon}{\sqrt{\mu}} \bm{\omega}_{z} \partial_{z} v + \frac{\epsilon}{\text{Ro} \sqrt{\mu}} \partial_{z} v.
\end{equation*}

\noindent Since $\bm{\omega}_{x} = -\frac{1}{\sqrt{\mu}} \partial_{z} v$ and $\bm{\omega}_{z} = \partial_{x} v$ we notice that $\epsilon \bm{\omega}_{x} \partial_{x} v + \frac{\epsilon}{\sqrt{\mu}} \bm{\omega}_{z} \partial_{z} v = 0$. Using Proposition \ref{equation_u} we get

\begin{equation*}
\partial_{t} \bm{\omega}_{y} + \epsilon \overline{u} \partial_{x} \bm{\omega}_{y} - \epsilon \partial_{x} \left[ \left(1 \hspace{-0.05cm} + \hspace{-0.05cm} z \hspace{-0.05cm} - \hspace{-0.05cm} \beta b \right) \hspace{-0.05cm} \overline{u} \right] \partial_{z}  \bm{\omega}_{y} - \frac{\epsilon}{\text{Ro} \sqrt{\mu}} \partial_{z} v + \epsilon \sqrt{\mu} A_{1} + \epsilon \mu A_{2} \hspace{-0.05cm} = \hspace{-0.05cm} \epsilon \hspace{-0.05cm} \max \hspace{-0.05cm} \left(\hspace{-0.05cm} \mu^{\frac{3}{2}},\beta \sqrt {\mu} \hspace{-0.05cm} \right) \hspace{-0.1cm} R,
\end{equation*}

\noindent where 

\begin{equation*}
\begin{aligned}
&A_{1} = u_{\text{sh}}^{\ast} \partial_{x} \bm{\omega}_{y} - \partial_{x} \left( \int_{-1+\beta b}^{z} u_{\text{sh}}^{\ast} \right) \partial_{z} \bm{\omega}_{y},\\
&A_{2} = - \frac{1}{2} \left( \hspace{-0.05cm} \left[1+z-\beta b\right]^{2} \hspace{-0.1cm} - \hspace{-0.1cm} \frac{h^{2}}{3} \hspace{-0.05cm} \right) \partial_{x}^{2} \overline{u} \partial_{x} \bm{\omega}_{y} \hspace{-0.1cm} + \hspace{-0.1cm} \frac{1}{2} \partial_{x} \left( \int_{-1+\beta b}^{z}  \left(\left[1+z-\beta b\right]^{2}- \frac{h^{2}}{3} \right) \partial_{x}^{2} \overline{u} \right) \partial_{z} \bm{\omega}_{y}.
\end{aligned}
\end{equation*}

\noindent Then, integrating with respect to $z$, using the fact that $\partial_{t} \zeta + \partial_{x} \left( h \overline{u} \right) = 0$ and $u_{\text{sh}} = - \int_{z}^{\epsilon \zeta} \bm{\omega}_{y}$, we get

\begin{equation*}
\begin{aligned}
\partial_{t} u_{\text{sh}} + \epsilon \overline{u} \partial_{x} u_{\text{sh}} + \epsilon u_{\text{sh}} \partial_{x} \overline{u} + \frac{\epsilon}{\text{Ro} \sqrt{\mu}} \left(\underline{v} - v \right) = &\epsilon \partial_{x} \left[ \left(1+z - \beta b \right) \overline{u} \right] \partial_{z}  u_{\text{sh}} + \epsilon \sqrt{\mu} \int_{z}^{\epsilon \zeta} A_{1}\\
& + \epsilon \mu \int_{z}^{\epsilon \zeta} A_{2} + \epsilon \max \left(\mu^{\frac{3}{2}},\beta \sqrt{\mu} \right) \hspace{-0.1cm} R.
\end{aligned}
\end{equation*}

\noindent Integrating again with respect to $z$, using the fact that $\partial_{t} \zeta + \partial_{x} \left( h \overline{u} \right) = 0$ and $Q_{x} = \overline{u_{\text{sh}}^{\ast}}$, we obtain

\begin{equation*}
\begin{aligned}
\partial_{t} Q_{x} + \epsilon \overline{u} \partial_{x} Q_{x} + \epsilon Q_{x} \partial_{x} \overline{u} + \frac{\epsilon}{\text{Ro} \sqrt{\mu}} \left(\underline{v} - \overline{v} \right) =& \epsilon \sqrt{\mu} \frac{1}{h} \int_{-1+\beta b}^{\epsilon \zeta} \int_{z}^{\epsilon \zeta} A_{1}\\
&\hspace{-0.2cm} + \hspace{-0.1cm} \epsilon \mu \frac{1}{h} \int_{-1+\beta b}^{\epsilon \zeta} \int_{z}^{\epsilon \zeta} A_{2} + \hspace{-0.05cm} \epsilon \hspace{-0.05cm} \max \left(\mu^{\frac{3}{2}},\beta \sqrt{\mu} \right) \hspace{-0.1cm} R.
\end{aligned}
\end{equation*}

\noindent The end of the proof is devoted to the computation of the others terms. We have

\begin{equation*}
\begin{aligned}
\int_{z}^{\epsilon \zeta} A_{1} &= \int_{z}^{\epsilon \zeta} u_{\text{sh}}^{\ast} \partial_{x} \bm{\omega}_{y} - \partial_{x} \left( \int_{-1+\beta b}^{z} u_{\text{sh}}^{\ast} \right) \partial_{z} \bm{\omega}_{y}\\
&= \int_{z}^{\epsilon \zeta} \partial_{x} \left(\uastsh \bm{\omega}_{y} \right) - \epsilon \zeta \text{Q}_{x} \underline{\bm{\omega}_{y}} + \partial_{x} \left( \int_{-1+\beta b}^{z} \uastsh \right) \bm{\omega}_{y}.
\end{aligned}
\end{equation*}

\noindent Since $\bm{\omega}_{y} = \partial_{z} \uastsh$, we obtain

\begin{equation*}
\int_{z}^{\epsilon \zeta} A_{1} = \text{Q}_{x} \partial_{x} \text{Q}_{x} - \uastsh \partial_{x} \uastsh + \partial_{x} \left( \int_{-1+\beta b}^{z} \uastsh \right) \partial_{z} \uastsh.
\end{equation*}

\noindent then, integrating  again with respect to $z$, we obtain

\begin{equation*}
\frac{1}{h} \int_{-1+\beta b}^{\epsilon \zeta} \int_{z}^{\epsilon \zeta} A_{1} = \text{Q}_{x} \partial_{x} \text{Q}_{x} - \frac{1}{h} \partial_{x} \int_{-1+\beta b}^{\epsilon \zeta} \left(\uastsh \right)^{2}.
\end{equation*}

\noindent Furthermore, we have

\begin{equation*}
\begin{aligned}
\int_{z}^{\epsilon \zeta} A_{2} &= - \frac{1}{2} \int_{z}^{\epsilon \zeta} \left( \left[1+z'-\beta b\right]^{2} - \frac{h^{2}}{3} \right) \partial_{x}^{2} \overline{u} \partial_{x} \bm{\omega}_{y}\\
&\hspace{0.4cm} + \frac{1}{2} \int_{z}^{\epsilon \zeta} \partial_{x} \left( \int_{-1+\beta b}^{z}  \left(\left[1+z'-\beta b\right]^{2}- \frac{h^{2}}{3} \right) \partial_{x}^{2} \overline{u} \right) \partial_{z} \bm{\omega}_{y}\\
&= - \frac{1}{2} \int_{z}^{\epsilon \zeta} \partial_{x} \left[ \left( \left[1+z'-\beta b\right]^{2} - \frac{h^{2}}{3} \right) \partial_{x}^{2} \overline{u} \bm{\omega}_{y} \right] - \epsilon \partial_{x} \zeta \frac{h^{2}}{3} \partial_{x}^{2} \overline{u} \underline{\bm{\omega}_{y}}\\
&\hspace{0.4cm} - \frac{1}{2} \partial_{x} \left(\int_{-1+\beta b}^{z} \left( \left[1+z'-\beta b\right]^{2} - \frac{h^{2}}{3} \right) \partial_{x}^{2} \overline{u} \right) \bm{\omega}_{y}.
\end{aligned}
\end{equation*}

\noindent Since $\bm{\omega}_{y} = \partial_{z} \uastsh$, we obtain

\begin{equation*}
\begin{aligned}
\int_{z}^{\epsilon \zeta} A_{2} = &\int_{z}^{\epsilon \zeta} \partial_{x} \left( \left[1+z' - \beta b \right] \partial_{x}^{2} \overline{u} \uastsh \right) + \frac{1}{2} \partial_{x} \left( \left( \left[1+z'-\beta b\right]^{2} - \frac{h^{2}}{3} \right) \partial_{x}^{2} \overline{u} \uastsh \right) \\
&- \frac{1}{2} \partial_{x} \left(\int_{-1+\beta b}^{z} \left( \left[1+z'-\beta b\right]^{2} - \frac{h^{2}}{3} \right) \partial_{x}^{2} \overline{u} \right) \partial_{z} \uastsh\\
&+ \frac{1}{3} \partial_{x} \left(h^{2} \partial_{x}^{2} \overline{u} \text{Q}_{x} \right) - \epsilon \partial_{x} \zeta h \partial_{x}^{2} \overline{u} \text{Q}_{x}.
\end{aligned}
\end{equation*}

\noindent Then we integrate again with respect to $z$ and we divide $h$. We obtain

\begin{equation*}
\begin{aligned}
 \frac{1}{h} \int_{-1+\beta b}^{\epsilon \zeta} \int_{z}^{\epsilon \zeta} A_{2} = & \frac{1}{h} \int_{-1+ \beta b}^{\epsilon \zeta} \int_{z}^{\epsilon \zeta} \partial_{x} \left( \left[1+z'-\beta b \right] \partial_{x}^{2} \overline{u} \uastsh \right)\\
 & + \frac{1}{2h} \int_{-1+ \beta b}^{\epsilon \zeta} \partial_{x} \left( \left(\left[1+z'-\beta b \right]^{2} - \frac{h^{2}}{3} \right) \partial_{x}^{2} \overline{u} \uastsh \right)\\
 & + \frac{1}{2h} \int_{-1+ \beta b}^{\epsilon \zeta} \partial_{x} \left( \left(\left[1+z'-\beta b \right]^{2} - \frac{h^{2}}{3} \right) \partial_{x}^{2} \overline{u} \right) \uastsh \\
 & + \frac{1}{3} \partial_{x} \left(h^{2} \partial_{x}^{2} \overline{u} \text{Q}_{x} \right) - \frac{4}{3} h \partial_{x} h \partial_{x}^{2} \overline{u} \text{Q}_{x} + \beta R.
\end{aligned}
\end{equation*}

\noindent Then, using the fact that

\begin{equation*}
\int_{-1+\beta b}^{\epsilon \zeta} \int_{z}^{\epsilon \zeta} \int_{-1+\beta b}^{z'} \partial_{x} \uastsh = \partial_{x} \int_{-1+\beta b}^{\epsilon \zeta} \int_{z}^{\epsilon \zeta} \int_{-1+\beta b}^{z'} \uastsh + \beta R,
\end{equation*}

\noindent we finally get

\begin{equation*}
\begin{aligned}
 \frac{1}{h} \int_{-1+\beta b}^{\epsilon \zeta} \int_{z}^{\epsilon \zeta} A_{2} = \frac{1}{3} \partial_{x} \left( h^{2} \text{Q}_{x} \partial_{x}^{2} \overline{u} \right) + \frac{1}{6} h^{2} u^{\sharp} \partial_{x}^{3} \overline{u} + \frac{1}{8h} \partial_{x} \left( h^{3} u^{\sharp} \right) \partial_{x}^{2} \overline{u} + \beta R,
\end{aligned}
\end{equation*}

\noindent and the first equation follows. The second equation follows similarly using the fact that $\uastsh = u_{\text{sh}} - \text{Q}_{x}$.
\end{proof}

\noindent We can also extend Proposition \ref{eq_Qy}.

\begin{prop}\label{eq_Qy_improve}
If \upshape $\left(\zeta, \Umuszero, \bm{\omega} \right)$ \itshape satisfy the Castro-Lannes system \eqref{castro_lannes_formulation}, then $\text{Q}_{x}$ satisfies the following equation

\upshape
\begin{equation*}
\begin{aligned}
\partial_{t} \text{Q}_{y}  + \epsilon \overline{u} \partial_{x} \text{Q}_{y} \hspace{-0.1cm} + \epsilon \text{Q}_{x} \partial_{x} \overline{v} + \frac{\epsilon}{\text{Ro} \sqrt{\mu}} \left( \overline{u}  -  \underline{u} \right) &= \epsilon \sqrt{\mu} \text{Q}_{x} \partial_{x} \text{Q}_{y} - \epsilon \sqrt{\mu} \frac{1}{3} h^{2} \partial_{x}^{2} \overline{u} \partial_{x} \overline{v}\\
&-\epsilon \sqrt{\mu} \frac{1}{h} \partial_{x} \left( \int_{-1+\beta b}^{\epsilon \zeta} \uastsh \vastsh \right)\\
& -\epsilon \mu \left(\partial_{x} h \right)^{2} \text{Q}_{x} \partial_{x} \overline{v} + \epsilon \mu \frac{h^{2}}{3} \partial_{x}^{2} \overline{u} \partial_{x} \text{Q}_{y} \\
& -\epsilon \mu \frac{1}{24 h}  \partial_{x}^{2} \hspace{-0.1cm} \left( \hspace{-0.05cm} h^{3} u^{\sharp} \hspace{-0.05cm} \right) \hspace{-0.1cm} \partial_{x} \overline{v} + \epsilon \mu \frac{1}{24 h} \partial_{x} \hspace{-0.1cm} \left( h^{3} v^{\sharp} \partial^{2}_{x} \overline{u} \right)\\ &+\epsilon \max \left(\mu^{\frac{3}{2}}, \beta \sqrt{\mu} \right) R,
\end{aligned}
\end{equation*}
\itshape

\noindent and $\vastsh$ satisfies the equation

\upshape
\begin{equation*}
\begin{aligned}
\partial_{t} \vastsh + \epsilon \overline{u} \partial_{x} \vastsh + \epsilon \uastsh \partial_{x} \overline{v} + \frac{\epsilon}{\text{Ro} \sqrt{\mu}} \left(u - \overline{u} \right) = & \epsilon \sqrt{\mu} \frac{1}{h} \partial_{x} \left( \int_{-1+\beta b}^{\epsilon \zeta} \uastsh \vastsh \right) - \epsilon \sqrt{\mu} \uastsh \partial_{x} \vastsh\\
& + \epsilon \partial_{x} \left(\int_{-1+\beta b}^{z} \left[\overline{u} + \sqrt{\mu} \uastsh \right] \right) \partial_{z} \vastsh\\ 
&+\epsilon \sqrt{\mu} \frac{1}{2} \left(\left[1+z - \beta b \right]^{2} - \frac{h^{2}}{3} \right) \partial_{x}^{2} \overline{u} \partial_{x} \overline{v} \\
&+ \epsilon \left(\mu, \beta \sqrt{\mu} \right) R.
\end{aligned}
\end{equation*}
\itshape

\end{prop}

\begin{proof}
\noindent Using the first equation of the vorticity equation of the Castro-Lannes system \eqref{castro_lannes_formulation}, we have

\begin{equation*}
\partial_{t} \bm{\omega}_{x} + \epsilon u \partial_{x} \bm{\omega}_{x} + \frac{\epsilon}{\mu} \text{w} \partial_{z}  \bm{\omega}_{x} = \epsilon \bm{\omega}_{x} \partial_{x} u + \frac{\epsilon}{\sqrt{\mu}} \bm{\omega}_{z} \partial_{z} u + \frac{\epsilon}{\text{Ro} \sqrt{\mu}} \partial_{z} u.
\end{equation*}

\noindent Then, using the fact that $\nabla^{\mu,0} \cdot \bm{\omega} = 0$ and $\nabla^{\mu,0} \cdot \textbf{U}^{\mu,\gamma}  = 0$, we get

\begin{equation*}
\partial_{t} \bm{\omega}_{x} - \frac{\epsilon}{\sqrt{\mu}} \partial_{z} \left( u \bm{\omega}_{z} \right) + \frac{\epsilon}{\mu} \partial_{z} \left( \text{w} \bm{\omega}_{x} \right) = \frac{\epsilon}{\text{Ro} \sqrt{\mu}} \partial_{z} u.
\end{equation*}

\noindent then, we integrate with respect to $z$ and, using the fact that $\partial_{t} \zeta - \frac{1}{\mu} \underline{\textbf{U}}^{\mu} \cdot \textbf{N}^{\mu,0} = 0$, $\bm{\omega}_{x} = - \frac{1}{\sqrt{\mu}} \partial_{z} v$ and $\bm{\omega}_{z} = \partial_{x} v$, we obtain

\begin{equation*}
\partial_{t} \left( \int_{-1+\beta b}^{\epsilon \zeta} \bm{\omega}_{x} \right) - \frac{\epsilon}{\sqrt{\mu}} \underline{u} \partial_{x} \underline{v} + \frac{\epsilon}{\sqrt{\mu}} u \partial_{x} v + \frac{\epsilon}{\mu^{\frac{3}{2}}} \text{w} \partial_{z} v + \frac{\epsilon}{\text{Ro} \sqrt{\mu}} \left(u - \underline{u} \right) = 0.
\end{equation*}

\noindent Then, we integrate again with respect to $z$ and, using Proposition \ref{equation_v} and the fact that $\partial_{t} \zeta - \frac{1}{\mu} \underline{\textbf{U}}^{\mu} \cdot \textbf{N}^{\mu,0} = 0$, $\textbf{U}_{b}^{\mu} \cdot \textbf{N}_{b}^{\mu,0} = 0$, and $\nabla^{\mu,0} \cdot \textbf{U}^{\mu} = 0$, we get

\begin{equation*}
\partial_{t} \text{Q}_{y} - \frac{\epsilon}{\sqrt{\mu}} \underline{u} \partial_{x} \underline{v} + \frac{\epsilon}{\sqrt{\mu}} \frac{1}{h} \partial_{x} \left( \int_{-1+\beta b}^{\epsilon \zeta} u v \right) + \frac{1}{\sqrt{\mu} h} \partial_{t} h \overline{v} + \frac{\epsilon}{\text{Ro} \sqrt{\mu}} \left(\overline{u} - \underline{u} \right) = 0.
\end{equation*}

\noindent Then, thanks to Propositions \ref{mean_eq}, \ref{equation_v} and \ref{equation_u}  we finally obtain that

\begin{equation*}
\begin{aligned}
\partial_{t} \text{Q}_{y} \hspace{-0.1cm} + \hspace{-0.1cm} \epsilon \overline{u} \partial_{x} \text{Q}_{y} \hspace{-0.1cm} + \hspace{-0.1cm} \epsilon \text{Q}_{x} \partial_{x} \overline{v} \hspace{-0.1cm} + \hspace{-0.1cm} \frac{\epsilon}{\text{Ro} \sqrt{\mu}} \hspace{-0.1cm} \left( \overline{u} \hspace{-0.1cm} - \hspace{-0.1cm} \underline{u} \right) \hspace{-0.1cm} &= \epsilon \sqrt{\mu} \text{Q}_{x} \partial_{x} \text{Q}_{y} - \epsilon \sqrt{\mu} \frac{1}{3} h^{2} \partial_{x}^{2} \overline{u} \partial_{x} \overline{v}\\
&-\epsilon \sqrt{\mu} \frac{1}{h} \partial_{x} \left( \int_{-1+\beta b}^{\epsilon \zeta} \uastsh \vastsh \right)\\
&+\epsilon \mu \overline{T \uastsh} \partial_{x} \overline{v} + \epsilon \mu \frac{h^{2}}{3} \partial_{x}^{2} \overline{u} \partial_{x} \text{Q}_{y}\\
&+ \epsilon \mu \frac{1}{2h} \partial_{x} \left( \int_{-1+\beta b}^{\epsilon \zeta} \vastsh \left(\left[1+z - \beta b \right]^{2} - \frac{h^{2}}{3} \right) \partial^{2}_{x} \overline{u} \right)\\
&+\epsilon \max \left(\mu^{\frac{3}{2}}, \beta \sqrt{\mu} \right) R.
\end{aligned}
\end{equation*}

\noindent Finally, we can compute that

\begin{equation*}
\frac{1}{2} \int_{-1+\beta b}^{\epsilon \zeta} \vastsh \left(\left[1+z - \beta b \right]^{2} - \frac{h^{2}}{3} \right) = \frac{1}{24} h^{3} v^{\sharp},
\end{equation*}

\noindent and the first equation follows from Lemma \ref{udiese_T_eq}. The second equation follows similarly using the fact that $\vastsh = v_{\text{sh}} - \text{Q}_{y}$.
\end{proof}

\noindent As noticed in \cite{Castro_Lannes_shallow_water}, the quantity $E$ defined by

\begin{equation}\label{def_E}
E = \begin{pmatrix} E_{xx} & E_{xy} \\ E_{xy} & E_{yy} \end{pmatrix} = \int_{-1+\beta b}^{\epsilon \zeta} \textbf{V}^{\ast}_{\text{sh}} \otimes \textbf{V}^{\ast}_{\text{sh}}
\end{equation}

\noindent appears in the equations of $\text{Q}_{x}$ and $\text{Q}_{y}$ and can not be express with respect to $\zeta$, $\overline{\textbf{V}}$ and $\textbf{V}^{\sharp}$. The following subsection is devoting to giving an equation for $E$.

\subsection{Equations for $E$}

\noindent In this part, we  derive an equation for $E$ up to terms of order $\mathcal{O} (\mu)$. We have to introduce the quantity $F$

\begin{equation}\label{def_F}
F = (F_{ijk})_{i,j,k} = \int_{-1+\beta b}^{\epsilon \zeta} \textbf{V}^{\ast}_{\text{sh}} \otimes \textbf{V}^{\ast}_{\text{sh}} \otimes \textbf{V}^{\ast}_{\text{sh}}.
\end{equation}

\noindent The following proposition gives an equation for $E$.

\begin{prop}\label{eq_E}
If \upshape $\left(\zeta, \Umuszero, \bm{\omega} \right)$ \itshape satisfy the Castro-Lannes system \eqref{castro_lannes_formulation}, then $E$ satisfies the following equation

\upshape
\small
\begin{equation*}
\begin{aligned}
\partial_{t} E \hspace{-0.05cm} + \hspace{-0.05cm} \epsilon \overline{u} \partial_{x} E + \epsilon l \left( E,\partial_{x} \overline{\textbf{V}} \right) \hspace{-0.05cm} + \hspace{-0.05cm} \epsilon \sqrt{\mu} \partial_{x} F_{\cdot,\cdot,1} \hspace{-0.05cm} + \hspace{-0.05cm} \frac{\epsilon}{\text{Ro}} E^{S} = &\left(\epsilon \sqrt{\mu} \partial_{x} \overline{v} + \frac{\epsilon \sqrt{\mu}}{\text{Ro}} \right) \mathcal{D} (\textbf{V}^{\sharp}, \overline{u} )\\
& + \max \left( \epsilon \mu,\epsilon \beta \sqrt{\mu}, \frac{\epsilon}{\text{Ro}} \mu \right) R, 
\end{aligned}
\end{equation*}
\normalsize
\itshape

\noindent where

\upshape
\begin{equation}\label{ES}
E^{S} = \int_{-1+\beta b}^{\epsilon \zeta} \textbf{V}_{\text{sh}}^{\perp} \otimes \textbf{V}_{\text{sh}}  + \textbf{V}_{\text{sh}} \otimes \textbf{V}_{\text{sh}}^{\perp} = \begin{pmatrix} -2 E_{xy} & E_{xx} - E_{yy} \\  E_{xx} - E_{yy} & 2 E_{xy} \end{pmatrix}
\end{equation}
\itshape

\medskip

\upshape
\begin{equation}\label{l_E_u}
l \left( E,\partial_{x} \overline{\textbf{V}} \right) = \begin{pmatrix} 3 \partial_{x} \overline{u} E_{xx} + 2 \partial_{x} \overline{v} E_{xy} & 2 \partial_{x} \overline{u} E_{xy} + \partial_{x} \overline{v} E_{yy} \\ 2 \partial_{x} \overline{u} E_{xy} + \partial_{x} \overline{v} E_{yy} & \partial_{x} \overline{u} E_{yy} \end{pmatrix}
\end{equation}
\itshape

\noindent and

\upshape
\begin{equation}\label{D_E}
 \mathcal{D} (\textbf{V}^{\sharp}, \overline{u} ) =  \partial_{x}^{2} \overline{u} \begin{pmatrix} 0 & u^{\sharp} \\  u^{\sharp} & 2 v^{\sharp} \end{pmatrix}.
\end{equation}
\itshape

\end{prop}

\begin{proof}
\noindent The proof is similar to the computation in Part 4.5.2 and Part 5.4.1 in \cite{Castro_Lannes_shallow_water}. We compute $\partial_{t} E$ and we use the second equations of Propositions \ref{eq_Qx_improve} and \ref{eq_Qy_improve} up to terms of order $\mathcal{O} (\mu)$. For the Coriolis contribution, we use the expansion of $u$ and $v$ given in Proposition \ref{equation_u} and \ref{equation_v}.
\end{proof}

\noindent The quantity $F$ appears in the equation of $E$ and can not be expressed with respect to $\zeta$, $\overline{\textbf{V}}$, $\textbf{V}^{\sharp}$ and $E$. The next proposition gives an equation for $F$ up to terms of order $\mathcal{O} (\sqrt{\mu})$.

\begin{prop}\label{eq_F}
If \upshape $\left(\zeta, \Umuszero, \bm{\omega} \right)$ \itshape satisfy the Castro-Lannes system \eqref{castro_lannes_formulation}, then $F_{ijk}$ satisfies the following equation

\upshape
\begin{equation*}
\begin{aligned}
\partial_{t} F_{ijk} + \epsilon ( \overline{u} \partial_{x} F_{ijk} + \partial_{x} \overline{u} F_{ijk} + F_{1kj} \partial_{x} \textbf{V}_{i} + F_{i1k} \partial_{x} \textbf{V}_{j} + F_{ij1} \partial_{x} \textbf{V}_{k} ) + &\frac{\epsilon}{\text{Ro}} F^{S} =\\
&\max \left( \epsilon , \frac{\epsilon}{\text{Ro}} \right) \sqrt{\mu} R,
\end{aligned}
\end{equation*}
\itshape

\noindent where

\upshape
\begin{equation}\label{FS}
F^{S} = \int_{-1+\beta b}^{\epsilon \zeta} \textbf{V}_{\text{sh}}^{\perp} \otimes \textbf{V}_{\text{sh}} \otimes \textbf{V}_{\text{sh}} + \textbf{V}_{\text{sh}} \otimes \textbf{V}_{\text{sh}}^{\perp} \otimes \textbf{V}_{\text{sh}} + \textbf{V}_{\text{sh}} \otimes \textbf{V}_{\text{sh}} \otimes \textbf{V}_{\text{sh}}^{\perp}. 
\end{equation}
\itshape

\end{prop}

\begin{proof}
\noindent The proof is similar to the computation in Part 4.5.3 and Part 5.4.2 in \cite{Castro_Lannes_shallow_water}. We compute $\partial_{t} F$ and we use the second equations of Propositions \ref{eq_Qx_improve} and \ref{eq_Qy_improve} up to terms of order $\mathcal{O} (\sqrt{\mu})$. For the Coriolis contribution, we use the expansion of $u$ and $v$ in Proposition \ref{equation_u} and \ref{equation_v}.
\end{proof}

\subsection{The Green-Naghdi equations}

\noindent We can now establish the Green-Naghdi equations when $d=1$. The Green-Naghdi equations are the following system

\small
\begin{equation}\label{green_naghdi}
\left\{
\begin{aligned}
&\partial_{t} \zeta + \partial_{x} \left(h \overline{u} \right) = 0,\\
&\hspace{-0.05cm} \left(\hspace{-0.05cm} 1 \hspace{-0.05cm} + \hspace{-0.05cm} \mu \mathcal{T} \hspace{-0.05cm} \right) \hspace{-0.05cm} \left(\partial_{t} \overline{u} \hspace{-0.05cm} + \hspace{-0.05cm} \epsilon \overline{u} \partial_{x} \overline{u} \hspace{-0.05cm} \right) \hspace{-0.05cm} + \hspace{-0.05cm} \partial_{x} \zeta \hspace{-0.05cm} - \hspace{-0.05cm} \frac{\epsilon}{\text{Ro}} \overline{v} \hspace{-0.05cm} + \hspace{-0.05cm} \epsilon \mu \mathcal{Q}(\overline{u}) \hspace{-0.05cm} +  \hspace{-0.05cm} \epsilon \mu \partial_{x} E_{xx} \hspace{-0.05cm} + \hspace{-0.05cm} \epsilon \mu ^{\frac{3}{2}} \mathcal{C}_{1} \hspace{-0.1cm} \left( u^{\sharp}, \overline{u} \right) + \frac{\epsilon}{\text{Ro}}\frac{\mu^{\frac{3}{2}}}{24 h} \partial_{x}^{2} (h^{3} v^{\sharp}) = 0,\\
&\partial_{t} \overline{v} + \epsilon \overline{u} \partial_{x} \overline{v} + \frac{\epsilon}{\text{Ro}} \overline{u} + \epsilon \mu \partial_{x} E_{xy} + \epsilon \mu ^{\frac{3}{2}} \mathcal{C}_{2} \left( v^{\sharp}, \partial_{x}^{2} \overline{u} \right) = 0,\\
&\partial_{t} \textbf{V}^{\sharp} + \epsilon \textbf{V}^{\sharp} \partial_{x} \overline{u} + \epsilon  \overline{u} \partial_{x} \textbf{V}^{\sharp} + \frac{\epsilon}{\text{Ro}} \textbf{V}^{\sharp \perp} = 0,\\
&\partial_{t} E \hspace{-0.05cm} + \hspace{-0.05cm} \epsilon \overline{u} \partial_{x} E + \epsilon \, l \! \left(E, \partial_{x} \overline{\textbf{V}} \right) \hspace{-0.05cm} + \epsilon \sqrt{\mu} \partial_{x} F_{\cdot, \cdot, 1} + \hspace{-0.05cm} \frac{\epsilon}{\text{Ro}} E^{S}  \hspace{-0.05cm} =  \hspace{-0.05cm} \left( \hspace{-0.05cm} \epsilon   \sqrt{\mu} \partial_{x} \overline{v}  \hspace{-0.05cm} +  \hspace{-0.05cm} \frac{\epsilon}{\text{Ro}}  \sqrt{\mu}  \hspace{-0.05cm} \right)  \hspace{-0.05cm} \mathcal{D} (\textbf{V}^{\sharp}  \hspace{-0.05cm}, \overline{u} ),\\
&\partial_{t} F_{ijk} + \epsilon \overline{u} \partial_{x} F_{ijk} + \epsilon \partial_{x} \overline{u} F_{ijk} + \epsilon F_{1kj} \partial_{x} \textbf{V}_{i} + \epsilon F_{i1k} \partial_{x} \textbf{V}_{j} + \epsilon F_{ij1} \partial_{x} \textbf{V}_{k} + \frac{\epsilon}{\text{Ro}} F^{S} = 0.
\end{aligned}
\right.
\end{equation}
\normalsize

\noindent where 

\begin{equation}\label{op_gn}
\begin{aligned}
&\mathcal{T} = - \frac{1}{3h} \partial_{x} \left(h^{3} \partial_{x} \cdot \right) ,\\
&\mathcal{Q}(\overline{u}) = \frac{2}{3h} \partial_{x} \left(h^{3} \left[ \partial_{x} \overline{u} \right]^{2} \right),\\
&\mathcal{C}_{1} \left( u^{\sharp}, \overline{u} \right) = - \frac{1}{6h} \partial_{x} \left(2 h^{3} u^{\sharp} \partial_{x}^{2} \overline{u} + \partial_{x} (h^{3} u^{\sharp}) \partial_{x} \overline{u} \right),\\
&\mathcal{C}_{2} \left( v^{\sharp}, w \right) = - \frac{1}{24h} \partial_{x} \left(h^{3} v^{\sharp} w \right),\\
&l \left( E,\partial_{x} \overline{\textbf{V}} \right) = \begin{pmatrix} 3 \partial_{x} \overline{u} E_{xx} + 2 \partial_{x} \overline{v} E_{xy} & 2 \partial_{x} \overline{u} E_{xy} + \partial_{x} \overline{v} E_{yy} \\ 2 \partial_{x} \overline{u} E_{xy} + \partial_{x} \overline{v} E_{yy} & \partial_{x} \overline{u} E_{yy} \end{pmatrix},\\
&\mathcal{D} (\textbf{V}^{\sharp}, \overline{u} ) =  \partial_{x}^{2} \overline{u} \begin{pmatrix} 0 & u^{\sharp} \\  u^{\sharp} & 2 v^{\sharp} \end{pmatrix}
\end{aligned}
\end{equation}

\noindent and 

\begin{equation}
\begin{aligned}
&E^{S} = \int_{-1+\beta b}^{\epsilon \zeta} \textbf{V}_{\text{sh}}^{\perp} \otimes \textbf{V}_{\text{sh}}  + \textbf{V}_{\text{sh}} \otimes \textbf{V}_{\text{sh}}^{\perp} = \begin{pmatrix} -2 E_{xy} & E_{xx} - E_{yy} \\  E_{xx} - E_{yy} & 2 E_{xy} \end{pmatrix},\\
&F^{S} = \int_{-1+\beta b}^{\epsilon \zeta} \textbf{V}_{\text{sh}}^{\perp} \otimes \textbf{V}_{\text{sh}} \otimes \textbf{V}_{\text{sh}} + \textbf{V}_{\text{sh}} \otimes \textbf{V}_{\text{sh}}^{\perp} \otimes \textbf{V}_{\text{sh}} + \textbf{V}_{\text{sh}} \otimes \textbf{V}_{\text{sh}} \otimes \textbf{V}_{\text{sh}}^{\perp},
\end{aligned}
\end{equation}

\noindent and $\textbf{V}^{\sharp}$ is defined in \eqref{u_diese}, $E$ in \eqref{def_E} and $F$ in \eqref{def_F}. Notice that the first, the second and the third equations of System \eqref{green_naghdi} are the classical Green-Naghdi equations with new terms due to the vorticity (terms with $\textbf{V}^{\sharp}$ and $E$). The last equations are important to get a close system. We can now state that the Green-Naghdi equations are an order $\mathcal{O}(\mu^{2})$ approximation of the water waves equations.

\begin{prop}
In the Green-Naghdi regime with small topography variations $\mathcal{A}_{\text{GN}}$, the Castro-Lannes equations \eqref{castro_lannes_formulation} are consistent at order $\mathcal{O}(\mu^{2})$ with the Green-Naghdi equations \eqref{green_naghdi} in the sense of Definition \ref{consistence}.
\end{prop}

\begin{proof}
\noindent The proof is similar to the one in Proposition \ref{constit_boussi}. The first equation of the Green-Naghdi equations is always satisfied for a solution of the Castro-Lannes formulation by Proposition \ref{mean_eq}. For the second equation, we use Proposition \ref{equation_u}, Proposition \ref{eq_Qx_improve} together with Proposition \ref{equation_partial_t_u}, Lemma \ref{udiese_T_eq} and Proposition \ref{udiese_eq}. Notice the fact that all the terms with $\text{Q}_{x}$ disappear. The third equation follows from Proposition \ref{equation_v}, \ref{equation_u} and \ref{eq_Qy_improve} (all the terms with $\text{Q}_{y}$ also disappear). The last equations follows from Propositions \ref{udiese_eq}, \ref{eq_E} and \ref{eq_F}.
\end{proof}

\begin{remark}
\noindent Notice that even without a Coriolis forcing, we can not decrease the number of equations in the previous Green-Naghdi equations. However, if one also suppose that the vorticity is initially of the form $\left(0,  \omega_{y}, 0 \right)^{t}$, which corresponds to the propagation of 2D water waves, we can significantly simplify the Green-Naghdi equations (See Section 4 in \cite{Castro_Lannes_shallow_water} and \cite{lannes_marche}).
\end{remark}

\subsection{A simplified model in the case of a weak rotation and medium amplitude waves}

\noindent As noticed in \cite{Castro_Lannes_shallow_water}, if we assume that $\epsilon = \mathcal{O}(\sqrt{\mu})$ we can simplify the Green-Naghdi equations. This regime corresponds to medium amplitude waves (in the terminology of \cite{Lannes_ww}). We also assume that $\frac{\epsilon}{\text{Ro}} = \mathcal{O}(\sqrt{\mu})$. Then, we can simplify the Green-Naghdi system \eqref{green_naghdi} by dropping all the terms of $\mathcal{O}(\mu^{2})$ and we get

\begin{equation}
\left\{
\begin{aligned}
&\partial_{t} \zeta + \partial_{x} \left(h \overline{u} \right) = 0,\\
&( 1 +  \mu \mathcal{T} ) \left( \partial_{t} \overline{u}  + \epsilon \overline{u} \partial_{x} \overline{u} \right) +  \partial_{x} \zeta - \frac{\epsilon}{\text{Ro}} \overline{v} + \epsilon \mu \mathcal{Q}(\overline{u}) +  \epsilon \mu \partial_{x} E_{xx} = 0,\\
&\partial_{t} \overline{v} + \epsilon \overline{u} \partial_{x} \overline{v} + \frac{\epsilon}{\text{Ro}} \overline{u} + \epsilon \mu \partial_{x} E_{xy} = 0,\\
&\partial_{t} E  +  \epsilon \overline{u} \partial_{x} E + \epsilon \, l \! \left(E, \partial_{x} \overline{\textbf{V}} \right) + \frac{\epsilon}{\text{Ro}} E^{S}   =  0.
\end{aligned}
\right.
\end{equation}

\noindent Notice that in this regime, we catch effects of the vorticity on $\overline{\textbf{V}}$ thanks to the second order tensor  $E$. Without vorticity, this regime is particularly interesting since it is related to the Camassa-Holm equation and the Degasperis-Procesi equation (see for instance \cite{constantin_lannes}). It could be interesting to understand how we can adapt these two scalar equations in presence of a Coriolis forcing.

\section*{Acknowledgments}

\noindent The author has been partially funded by the ANR project Dyficolti ANR-13-BS01-0003.

\footnotesize
\bibliographystyle{plain}
\bibliography{biblio}
\normalsize

\end{document}